\documentclass{article}
\usepackage[letterpaper, margin=1in]{geometry}
\usepackage{amsmath}
\usepackage{faktor}
\usepackage{amssymb}
\usepackage{listings}
\usepackage{courier}
\usepackage{mdframed}
\usepackage{fancyhdr}
\usepackage{enumitem}
\usepackage{graphicx}
\graphicspath{{images/}}
\usepackage[utf8]{inputenc}
\usepackage[english]{babel}
\usepackage{amsthm}
\usepackage[T1]{fontenc}
\usepackage{tikz-cd}

\newenvironment{psmallmatrix}
  {\left(\begin{smallmatrix}}
  {\end{smallmatrix}\right)}

\theoremstyle{plain}
\newtheorem{thm}{Theorem}[section]
\newtheorem{lem}[thm]{Lemma}
\newtheorem{prop}[thm]{Proposition}
\newtheorem{fact}[thm]{Fact}
\newtheorem{cor}[thm]{Corollary}

\theoremstyle{defn}
\newtheorem{defn}[thm]{Definition}

\theoremstyle{remark}
\newtheorem*{remark}{Remark}

\usepackage{mathabx}
\usepackage{color}

\usepackage{color,soul}
\usepackage{enumitem}
 
\begin{document}

\begin{titlepage}
    \begin{center}
        \vspace*{1cm}
        \Huge
\textbf{The Freyd-Mitchell Embedding Theorem}
        \vspace{3.5cm}
        \Large
        
        \textbf{Arnold Tan Junhan} 
        \vspace{2.0cm}
        \Large
        
        \textbf{Michaelmas 2018 Mini Projects: Homological Algebra}
        \vfill
        \vspace{0.8cm}
        \Large
        University of Oxford   
    \end{center}
\end{titlepage}  
\tableofcontents
%we should have consistent capitalisation of section names
\section{Abstract}

Given a small abelian category $\mathcal{A}$, the \textit{Freyd-Mitchell embedding theorem} states the existence of a ring $R$ and an exact full embedding $\mathcal{A} \rightarrow R$-Mod, $R$-Mod being the category of left modules over $R$. This theorem is useful as it allows one to prove general results about abelian categories within the context of $R$-modules. The goal of this report is to flesh out the proof of the embedding theorem.
\\

\noindent We shall follow closely the material and approach presented in Freyd (1964).
\\ This means we will encounter such concepts as projective generators, injective cogenerators, the Yoneda embedding, injective envelopes, Grothendieck categories, subcategories of mono objects and subcategories of absolutely pure objects. This approach is summarised as follows:
\begin{itemize}
    \item the functor category $[\mathcal{A}, Ab]$ is abelian and has injective envelopes.
    \item in fact, the same holds for the full subcategory $\mathcal{L}(\mathcal{A})$ of left-exact functors.
    \item $\mathcal{L}(\mathcal{A})^{op}$ has some nice properties: it is cocomplete and has a projective generator.
    \item such a category embeds into $R$-Mod for some ring $R$.
    \item in turn, $\mathcal{A}$ embeds into such a category.
\end{itemize}
\section{Basics on abelian categories}

Fix some category $\mathcal{C}$. Let us say that a monic $A \rightarrow B$ is \textbf{contained} in another monic $A' \rightarrow B$ if there is a map $A \rightarrow A'$ making the diagram

 \[
  \begin{tikzcd}[row sep=0.4em,column sep=2em]
    A \arrow[dd] \arrow[dr]  \\
    & B  & \text{commute.}\\
    A' \arrow[ur] & \\
  \end{tikzcd}
 \]
 
\noindent We declare two monics $A \rightarrow B$ and $A' \rightarrow B$ to be \textbf{equivalent} if each is contained in the other. In this case $A$ and $A'$ are isomorphic.
\\ A \textbf{subobject} of $B$ is an equivalence class of monics into $B$. The relation of containment gives a partial ordering on subobjects.
\\

\noindent Dually, let us declare two epics $B \rightarrow C$ and $B \rightarrow C'$ in $\mathcal{C}$ to be \textbf{equivalent} if there are maps $C \rightarrow C'$ and $C' \rightarrow C$ such that

 
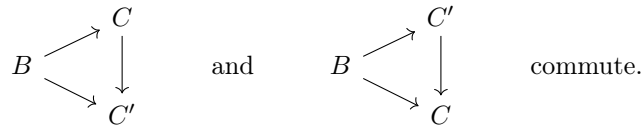
\[
  \begin{tikzcd}[row sep=0.4em,column sep=2em]
    & C \arrow[dd]  &&& C' \arrow[dd]  \\
    B \arrow[ur] \arrow[dr] &&\text{and}& B \arrow[ur] \arrow[dr]&& \text{commute.}\\
    & C' &&& C \\
  \end{tikzcd}
 \]
 
\noindent A \textbf{quotient object} of $B$ is an equivalence class of epics out of $B$, and we say the quotient object represented by $B \rightarrow C$ is \textbf{smaller} than that represented by $B \rightarrow C'$ if we have just the right diagram above.

\noindent If two quotient objects $B \rightarrow C$, $B \rightarrow C'$ are equivalent, then $C \cong C'$.
\\

\noindent When the context is clear, we will often just say \textit{$A$ is a subobject of $B$}, or \textit{$C$ is a quotient object of $B$}.
\\

\begin{defn}
A category is \textbf{complete} if every pair of maps has an equaliser, and every indexed set of objects has a product. Dually, a category is \textbf{cocomplete} if every pair of maps has a coequaliser, and every indexed set of objects has a sum.
\\ A category is \textbf{bicomplete} if it is both complete and cocomplete.
\\
\end{defn}

\begin{defn}
A category $\mathcal{A}$ is \textbf{abelian} if

\begin{itemize}
    \item[A0.] $\mathcal{A}$ has a zero object.
    \\
    \item[A1.] For every pair of objects there is a product and
    \item[A1*.] a sum.
    \\
    \item[A2.] Every map has a kernel and
    \item[A2*.] a cokernel.
    \\
    \item[A3.] Every monomorphism is a kernel of a map.
    \item[A3*.] Every epimorphism is a cokernel of a map.
    \\
    \\
\end{itemize}
\end{defn}

\noindent Let $A$ be an object in abelian category $\mathcal{A}$. Let $\mathcal{S}$ and $\mathcal{Q}$ be the families of subobjects and quotient objects of $A$, respectively. Define two functions $Cok: \mathcal{S} \rightarrow \mathcal{Q}$ and $Ker: \mathcal{Q} \rightarrow \mathcal{S}$, where $Cok$ assigns to each subobject its cokernel, and $Ker$ assigns to each quotient object its kernel. Note that these are order-reversing functions. For instance, if the monic $A \rightarrow B$ is contained in the monic $A' \rightarrow B$, then the epic $Cok(A' \rightarrow B)$ is smaller than the epic $Cok(A \rightarrow B)$.
\\

\begin{thm} \label{2.11}
For each $A$ in an abelian category $\mathcal{A}$, $Ker$ and $Cok$ are mutually inverse functions.
\end{thm}
\begin{proof}
We note in passing that $Ker$ and $Cok$ really are well-defined, as a kernel is always monic and a cokernel always epic.
\\ Now let $A' \rightarrow A$ be a monic. Let $A \rightarrow F$ be the cokernel of $A' \rightarrow A$, and $K \rightarrow A$ the kernel of $A \rightarrow F$. We must show that $K \rightarrow A$ is the same subobject as $A' \rightarrow A$.
\\ By Axiom A3 $A' \rightarrow A$ is the kernel of some $A \rightarrow B$.
\begin{itemize}
    \item $A' \rightarrow A \rightarrow F = 0$, so $A' \rightarrow A$ factors through the kernel of $A \rightarrow F$ as $A' \rightarrow K \rightarrow F$.
    \item On the other hand, $A' \rightarrow A \rightarrow B = 0$ so $A \rightarrow B$ factors through the cokernel of $A' \rightarrow A$ as $A \rightarrow F \rightarrow B$.
\\ Therefore $K \rightarrow A \rightarrow B =K \rightarrow A \rightarrow F \rightarrow B = 0$, and $K \rightarrow A$ factors through the kernel of $A \rightarrow B$ as $K \rightarrow A' \rightarrow B$.
\end{itemize}
We have shown that $KerCok$ is identity. Dually, $CokKer$ is identity.
\\
\end{proof}

\begin{thm}[Abelian categories are balanced]
In an abelian category, every monic epic map is an isomorphism.
\\
\end{thm}

\begin{proof}
Let $A \xrightarrow{x} B$ be monic and epic.
\\ Obviously, $B \rightarrow 0 = Cok(x)$, so by the result above, $x = Ker(B \rightarrow 0)$.
\\ $B \xrightarrow{1_B} B$ factors through the kernel $x$ of $B \rightarrow 0$: there is $B \xrightarrow{y} A$ with $xy = B \xrightarrow{1_B} B$. ($x$ is \textit{split epic}.)
\noindent Then $xyx = 1_Bx=x1_A$, and since $x$ is monic, $yx=1_A$.
\\
\end{proof}

\noindent The \textbf{intersection} of two subobjects of $\mathcal{A}$ is their greatest lower bound in the family of subobjects of $A$, with respect to containment.

\begin{thm}
In an abelian category, every pair of subobjects has an intersection.
\end{thm}

\begin{proof}
Let $A_1 \xrightarrow{f_1} A$ and $A_2 \xrightarrow{f_2} A$ be monics, $A \xrightarrow{c} F$ a cokernel of $A_1 \xrightarrow{f_1} A$, and $A_{12} \xrightarrow{k} A_2$ a kernel of $A_2 \xrightarrow{f_2} A \xrightarrow{c} F$.
\\ By definition of $k$, we have $A_{12}\xrightarrow{f_2k} A \xrightarrow{c} F = 0$. Since $A_1 \xrightarrow{f_1} A$ is a kernel of $A \xrightarrow{c} F$, $f_2k$ factors as $A_{12} \xrightarrow{m} A_1 \xrightarrow{f_1} A$. $m$ must be monic, since $f_1m=f_2k$ is monic as the composition of two monics.
\\
We therefore have a commutative diagram
\[ \begin{tikzcd}
A_{12} \arrow{r}{k} \arrow[swap]{d}{m} & A_2 \arrow{d}{f_2} \\
A_1 \arrow{r}{f_1} & A
\end{tikzcd} 
\]
which we claim is actually a pullback square: for each commutative diagram
\[ \begin{tikzcd}
X \arrow{r}{x_2} \arrow[swap]{d}{x_1} & A_2 \arrow{d}{f_2} \\
A_1 \arrow{r}{f_1} & A
\end{tikzcd} 
\]
there exists a unique $X \xrightarrow{x} A_{12}$ such that $mx=x_1$ and $kx=x_2$.
\\

\noindent Well, we have $(cf_2)x_2 = cf_1x_1=0$, so $x_2$ factors uniquely through the kernel $A_{12} \xrightarrow{k} A_2$ of $cf_2$: $$X \xrightarrow{x_2} A_2 = X \xrightarrow{x} A_{12} \xrightarrow{k} A_2,$$
where $x$ is unique such that $kx = x_2$. It only remains to see that $mx=x_1$. This is because
$$f_1mx=f_2kx=f_2x_2=f_1x_1,$$
and $f_1$ is monic.
\\

\noindent In particular, $A_{12} \xrightarrow{f_1m \ = \ f_2k} A$ is the intersection of $A_1 \xrightarrow{f_1} A$ and $A_2 \xrightarrow{f_2} A$, since when $X$ is a subobject contained in $A_1$ and $A_2$, $X$ will also be contained in $A_{12}$.
\end{proof}

\noindent Dually, any two quotient objects have a greatest lower bound. Since $Ker$ and $Cok$ are order-reversing and mutually inverse, every pair of subobjects has a least upper bound: for maps $A_i \rightarrow B$, $(i=1,2)$, their cokernels $B \rightarrow C_i$ have a least upper bound $B \rightarrow C_0$. Take the kernels:

 \[
  \begin{tikzcd}[row sep=0.8em,column sep=2em]
    C_1 \arrow[d] &&& A_1 \arrow[d] \arrow[dr] \\
   C_0 & B \arrow{l} \arrow{ul} \arrow{dl}  &\text{$\xrightarrow[\text{kernels}]{\text{take}}$}& A_0 \arrow{r} & B & \\
    C_2 \arrow[u] &&& A_1\arrow[ur] \arrow{u}\\
  \end{tikzcd}
 \]
\noindent Hence the family of subobjects of $A$ is a lattice. We write $\bigcap$ for the greatest lower bound operation (\textit{meet}) and $\bigcup$ for the least upper bound operation (\textit{join}).
\\

\begin{fact}
Abelian categories have all equalisers and all pullbacks.
Dually, abelian categories have all coequalisers and all pushouts.
\end{fact}

\noindent For instance, then, when we want to show that an abelian category is complete, we just need to check that it has all products.
\\
\begin{defn}
The \textbf{image} $Im(A \rightarrow B)$ of a map $A \rightarrow B$ is the smallest subobject of $B$ such that $A \rightarrow B$ factors through the representing monics.
\\ Dually, the \textbf{coimage} $Coim(A \rightarrow B)$ of $A \rightarrow B$ is the smallest quotient object of $A$ through which $A \rightarrow B$ factors.
\end{defn}

\noindent Recall that $Ker$ and $Cok$ were mutually inverse on subobjects and quotients, but we may of course take the $Ker$ and $Cok$ of any map, then $KerCok$ and $CokKer$ need not be identity. In fact:

\begin{fact}
In an abelian category,
\begin{itemize}
    \item $A \rightarrow B$ has an image, namely, $KerCok(A \rightarrow B)$.
    \item $A \xrightarrow{x} B$ is epic iff $Im(x)=B$, and hence, iff $Cok(x) = 0$.
    \item $A \xrightarrow{x} Im(x)$ is epic.
\end{itemize}
Dually,
\begin{itemize}
    \item $A \rightarrow B$ has a coimage, namely, $CokKer(A \rightarrow B)$.
    \item $A \xrightarrow{x} B$ is monic iff $Coim(x)=A$, and hence, iff $Ker(x) = 0$.
    \item $Coim \xrightarrow{x} B$ is monic.
    \\
\end{itemize}
\end{fact}

\noindent Next, we state a couple of lemmas for abelian categories:

\begin{lem} \label{2.64}
Suppose we have exact columns and exact middle row in the following commutative diagram:
\[ \begin{tikzcd}[row sep=0.8em,column sep=1em]
 &0 \arrow{d} &0 \arrow{d} &0 \arrow{d}\\
0 \arrow{r} & B_{11} \arrow{r} \arrow[swap]{d} & B_{12} \arrow{r} \arrow{d} & B_{13} \arrow{d}  \\
0 \arrow{r} & B_{21} \arrow{r} \arrow[swap]{d} & B_{22} \arrow{r} \arrow{d} & B_{23}  \\
0 \arrow{r} & B_{31} \arrow{r} \arrow{d} & B_{32} \arrow{d}  \\
&0 &0 \\
\end{tikzcd}
\]
Then the bottom row is exact iff the top row is exact.
\end{lem}

\begin{proof}
First, we prove the forward direction.
\begin{itemize}
    \item $Ker(B_{11} \rightarrow B_{12}) = 0:$
    \\ Let $A \rightarrow B_{11} \rightarrow B_{12}=0$. Then $A \rightarrow B_{11} \rightarrow B_{21} \rightarrow B_{22} = A \rightarrow B_{11} \rightarrow B_{12} \rightarrow B_{22}=0$.
    \\ Hence $A \rightarrow B_{11} \rightarrow B_{21}$ factors through $Ker(B_{21} \rightarrow B_{22})=0$. Hence $A \rightarrow B_{11}$ factors through $Ker(B_{11} \rightarrow B_{21})=0$, hence $A \rightarrow B_{11} = 0$.
    
    \item $Im(B_{11} \rightarrow B_{12}) \subset Ker(B_{12} \rightarrow B_{13}) :$
    \\ It is enough to see that $B_{11} \rightarrow B_{12}$ factors through $Ker(B_{12} \rightarrow B_{13})$, i.e., $B_{11} \rightarrow B_{12} \rightarrow B_{13} = 0$.
    \\ This follows because $B_{11} \rightarrow B_{12} \rightarrow B_{13} \rightarrow B_{23} =B_{11} \rightarrow B_{21} \rightarrow B_{22} \rightarrow B_{23}= 0$, and $B_{13} \rightarrow B_{23}$ is monic.
    
    \item $ Ker(B_{12} \rightarrow B_{13}) \subset Im(B_{11} \rightarrow B_{12}) :$
    \\ We show that whenever $L \rightarrow B_{12} \rightarrow B_{13} =0$, $L \rightarrow B_{12}$ factors through $B_{11} \rightarrow B_{12}$.
    \\ Well, $0 = L \rightarrow B_{12} \rightarrow B_{13} \rightarrow B_{23} = L \rightarrow B_{12} \rightarrow B_{22} \rightarrow B_{23}$,
    \\ so $L \rightarrow B_{12} \rightarrow B_{22} \subset Ker(B_{22} \rightarrow B_{23}) = B_{21} \rightarrow B_{22}$.
    \\ That is, there is some $L \rightarrow B_{21}$ such that $L \rightarrow B_{12} \rightarrow B_{22} = L \rightarrow B_{21} \rightarrow B_{22}$.
    \\ Next, we have $L \rightarrow B_{21} \rightarrow B_{31} \rightarrow B_{32} =L \rightarrow B_{12} \rightarrow B_{22} \rightarrow B_{32} =0$, so $L \rightarrow B_{21} \rightarrow B_{31} = 0$.
    \\ Hence $L \rightarrow B_{21} \subset Ker(B_{21} \rightarrow B_{31}) = B_{11} \rightarrow B_{21}$, and we factor $L \rightarrow B_{21} = L \rightarrow B_{11} \rightarrow B_{21}$.
    \\ Then $L \rightarrow B_{11} \rightarrow B_{12} \rightarrow B_{22} = L \rightarrow B_{11} \rightarrow B_{21} \rightarrow B_{22} = L \rightarrow B_{21} \rightarrow B_{22} = L \rightarrow B_{12} \rightarrow B_{22},$
    \\ and we are done since $B_{12} \rightarrow B_{22}$ is monic.
\end{itemize}
For the other direction, we need only show that $Ker(B_{31} \rightarrow B_{32}) = 0$. This can be chased similarly, but we note it follows immediately by the snake lemma applied to the top two rows (after replacing $B_{13}$ with $I$, where $I \rightarrow B_{13} = Im(B_{12} \rightarrow B_{13})$).
\end{proof}

\begin{lem}[Nine Lemma] \label{2.65}
Suppose we have exact columns and exact middle row in the commutative diagram:
\[ \begin{tikzcd}[row sep=0.8em,column sep=1em]
 &0 \arrow{d} &0 \arrow{d} &0 \arrow{d}\\
0 \arrow{r} & B_{11} \arrow{r} \arrow[swap]{d} & B_{12} \arrow{r} \arrow{d} & B_{13} \arrow{d} \arrow{r} & 0  \\
0 \arrow{r} & B_{21} \arrow{r} \arrow[swap]{d} & B_{22} \arrow{r} \arrow{d} & B_{23} \arrow{r} \arrow{d} & 0  \\
0 \arrow{r} & B_{31} \arrow{r} \arrow{d} & B_{32} \arrow{d}  \arrow{r} & B_{33} \arrow{d} \arrow{r} &0\\
&0 &0 &0\\
\end{tikzcd}
\]
Then the bottom row is exact iff the top row is exact.
\end{lem}

\begin{proof}
This follows immediately from Lemma \ref{2.64} and its dual.
%%%%%%%%%%%%%%%%%
\\
\end{proof}

\noindent Recall that the direct sum $A \oplus B$ plays the role of the (binary) categorical sum and product in an abelian category:
\begin{itemize}
    \item We have projection maps $A \oplus B \xrightarrow{\pi_1} A$, $A \oplus B \xrightarrow{\pi_2} B$. Two maps $C \xrightarrow{f_1} A$, $C \xrightarrow{f_2} B$ define a unique map $C \xrightarrow{\langle f_1,f_2 \rangle} A \oplus B$ such that $\pi_j \circ \langle f_1,f_2 \rangle = f_j$.
    \item We have inclusion maps $A \xrightarrow{\iota_1} A \oplus B$, $B \xrightarrow{\iota_2} A \oplus B$. Two maps $A \xrightarrow{g_1} C$, $B \xrightarrow{g_2} C$ define a unique map $A \oplus B \xrightarrow{[ g_1,g_2]} C$ such that $[ g_1,g_2 ] \circ \iota_j = g_j$.
    \\
\end{itemize}
\noindent We may add two maps $f,g:A \rightarrow B$ by defining $f+g: A \rightarrow B$ to be the map $A \xrightarrow{\Delta = \langle 1,1 \rangle} A \oplus A \xrightarrow{[f,g]} B$. Alternatively, we could define it as $A \xrightarrow{ \langle f,g \rangle} B \oplus B \xrightarrow{\Sigma = [1,1]} B$. Both have the zero map $A \xrightarrow{0} B$ as a unit, so an Eckmann-Hilton argument shows that the two operations are the same, and in fact associative and commutative. In fact:

\begin{thm}
The set $Hom(A,B)$ with the operation $+$ is an abelian group.
\end{thm}

\begin{proof}
It remains to exhibit an inverse $A \xrightarrow{-x} B$ for $A \xrightarrow{x} B$.
\\ It is convenient to introduce matrix notations for maps to/from the direct sum:
\\ Write $\begin{psmallmatrix} w \\y\end{psmallmatrix}$ for $[w,y]$, $\begin{psmallmatrix} w & x\end{psmallmatrix}$ for $\langle w, x \rangle$, and $\begin{psmallmatrix}w & x\\y & z\end{psmallmatrix}$ for $[\langle w,x \rangle, \langle y,z \rangle] = \langle [w,y], [x,z] \rangle$. Then a map $\begin{psmallmatrix}w & x\\y & z\end{psmallmatrix} \circ \begin{psmallmatrix} p & q\\r & s \end{psmallmatrix}$ is computed as the matrix product $\begin{psmallmatrix}p & q \\ r &  s \end{psmallmatrix} \begin{psmallmatrix}w & x\\y & z\end{psmallmatrix}$.
\\

\noindent Define a map $A \oplus B \xrightarrow{\begin{psmallmatrix}1 & x\\0 & 1\end{psmallmatrix}} A \oplus B$.
\\ The kernel of $\begin{psmallmatrix}1 & x\\0 & 1\end{psmallmatrix}$ is $K \xrightarrow{\begin{psmallmatrix}a & b\end{psmallmatrix}} A \oplus B$ where $0 = K \xrightarrow{\begin{psmallmatrix}a & b\end{psmallmatrix}} A \oplus B \xrightarrow{\begin{psmallmatrix}1 & x\\0 & 1\end{psmallmatrix}} A \oplus B = K \xrightarrow{\begin{psmallmatrix}a & xa+b \end{psmallmatrix}} A \oplus B$, so $a=b=0$.
\\ This shows that $\begin{psmallmatrix}1 & 0\\0 & 1\end{psmallmatrix}$ is monic. Dually, it is epic, hence it has an inverse map $\begin{psmallmatrix}p & q\\r & s\end{psmallmatrix}$.
\\ Since $\begin{psmallmatrix}1 & x\\0 & 1\end{psmallmatrix} \begin{psmallmatrix}p & q\\r & s\end{psmallmatrix} = \begin{psmallmatrix}1 & 0\\0 & 1\end{psmallmatrix}$, we conclude in particular that $q+x = 0$.
\\
\end{proof}

\noindent This upgrades the representable $Hom(A,-)$ (for each $A \in \mathcal{A}$) from a functor $\mathcal{A} \rightarrow Set$ to a functor $\mathcal{A} \rightarrow Ab$, where $Ab$ is the category of abelian groups.
\\

\begin{remark} \label{2.42}
We know that the direct sum $A \oplus B$ is unique up to isomorphism, and may be characterised as a system $A \overset{\iota_1}{\underset{\pi_1}\rightleftarrows} X \overset{\iota_2}{\underset{\pi_2}\leftrightarrows} B $ where $\pi_1 \iota_1 = 1_A, \pi_2 \iota_2 = 1_B, \pi_1 \iota_2=\pi_2 \iota_1=0$, and $\iota_1 \pi_1 + \iota_2 \pi_2 = 1_X$.
\\ Equivalently, it is a system $A \overset{\iota_1}{\underset{\pi_1}\rightleftarrows} X \overset{\iota_2}{\underset{\pi_2}\leftrightarrows} B $ where $\pi_1 \iota_1 = 1_A, \pi_2 \iota_2 = 1_B$, and $A \xrightarrow{\iota_1} X \xrightarrow{\pi_2} B$, $B \xrightarrow{\iota_2} X \xrightarrow{\pi_1} A$ are exact.
\end{remark}

%%%%%%%%%%%%%
\section{Additives and representables}

To any functor $F: \mathcal{A} \rightarrow \mathcal{B}$ is associated a function $Hom(A_1,A_2) \rightarrow Hom(FA_1,FA_2)$.
\\ If $\mathcal{A}$ and $\mathcal{B}$ are abelian categories, we say $F$ is \textbf{additive} if this function is a group homomorphism (with respect to $+$) for every $A_1, A_2 \in \mathcal{A}$.
\\
The functors $Hom(A,-) : \mathcal{A} \rightarrow Ab$ and $Hom(-,A): \mathcal{A}^{op} \rightarrow Ab$ are additive, because they are left-exact (see Corollary \ref{3.12}).
\\

\begin{thm}
A functor between abelian categories is additive iff it carries direct sums into direct sums.
\end{thm}

\begin{proof}
Suppose $A \overset{\iota_1}{\underset{\pi_1}\rightleftarrows} X \overset{\iota_2}{\underset{\pi_2}\leftrightarrows} B $ is a direct sum system in $\mathcal{A}$ (so $\pi_1 \iota_1 = 1_A, \pi_2 \iota_2 = 1_B, \pi_1 \iota_2=\pi_2 \iota_1=0$, and $\iota_1 \pi_1 + \iota_2 \pi_2 = 1_X$).
\\

\noindent Applying a functor $F: \mathcal{A} \rightarrow \mathcal{B}$ yields a direct sum system in $\mathcal{B}$, if $F$ is additive.
\\ Conversely, suppose applying $F: \mathcal{A} \rightarrow \mathcal{B}$ yields a direct sum system in $\mathcal{B}$ . Let us show that $F(x+y)=F(x)+F(y)$ for any $x,y: A \rightarrow B$.
\\ By definition, $A \xrightarrow{x+y} B = A \xrightarrow{\begin{psmallmatrix}1 & 1\end{psmallmatrix}} A \oplus A \xrightarrow{\begin{psmallmatrix} x\\y\end{psmallmatrix}} B$, so
$$F(A \xrightarrow{x+y} B) = FA \xrightarrow{F\begin{psmallmatrix}1 & 1\end{psmallmatrix}} F(A \oplus A) \xrightarrow{F\begin{psmallmatrix} x\\y\end{psmallmatrix}} FB = FA \xrightarrow{\begin{psmallmatrix}1 & 1\end{psmallmatrix}} F(A \oplus A) \xrightarrow{\begin{psmallmatrix} Fx\\Fy\end{psmallmatrix}} FB = FA \xrightarrow{Fx+Fy} FB .$$
\end{proof}

\noindent Working over an abelian category $\mathcal{A}$, let us call a sequence $\cdots \rightarrow A_1 \rightarrow A_2 \rightarrow A_3 \rightarrow \cdots$ \textbf{exact} if for each $i$, the kernel of $A_i \rightarrow A_{i+1}$ equals the image of $A_{i-1} \rightarrow A_i$ as subobjects of $A_i$.
\\ An exact sequence of the form $0 \rightarrow A' \rightarrow A \rightarrow A''$ is \textbf{left-exact}, and one of the form $A' \rightarrow A \rightarrow A'' \rightarrow 0$ is \textbf{right-exact}.
\\ We say a functor between abelian categories is \textbf{left-exact} if it carries left-exact sequences into left-exact sequences, \textbf{right-exact} if it carries right-exact sequences into right-exact sequences, and \textbf{exact} if it is both.

\begin{cor} \label{3.12}
Any left-exact or right-exact functor is additive.
\\
\end{cor}

\begin{proof}
If $A \overset{\iota_1}{\underset{\pi_1}\rightleftarrows} X \overset{\iota_2}{\underset{\pi_2}\leftrightarrows} B $ is a direct sum system in $\mathcal{A}$ (so $\pi_1 \iota_1 = 1_A, \pi_2 \iota_2 = 1_B$, and $A \xrightarrow{\iota_1} X \xrightarrow{\pi_2} B$, $B \xrightarrow{\iota_2} X \xrightarrow{\pi_1} A$ are exact), then these conditions are preserved by left-exact or right-exact functors.
\end{proof}

\noindent Let us say that a functor $F: \mathcal{A} \rightarrow \mathcal{B}$ is \textbf{faithful}, or an \textbf{embedding}, if for any $A_1, A_2 \in \mathcal{A}$ we have that the function $Hom(A_1,A_2) \rightarrow Hom(FA_1,FA_2)$ is injective.
\\

\begin{lem} \label{2.21}
For $A \rightarrow B \rightarrow C$ the following conditions are equivalent:
\begin{enumerate}
    \item $Im(A \rightarrow B) = Ker(B \rightarrow C)$;
    \item $Cok(A \rightarrow B) = Coim(B \rightarrow C)$;
    \item $A \rightarrow B \rightarrow C = 0$ and $K \rightarrow B \rightarrow F =0$,
\end{enumerate}
where $K \rightarrow B$ is a kernel of $B \rightarrow C$, and $B \rightarrow F$ is a cokernel of $A \rightarrow B$.
\end{lem}

\begin{proof}
We prove equivalence of the first and third items; equivalence of the second and third is proven dually.
\begin{itemize}
    \item The first item implies the third:
    \\ $A \rightarrow B \rightarrow C = A \rightarrow Im(A \rightarrow B) \rightarrow B \rightarrow C = A \rightarrow Ker(B \rightarrow C) \rightarrow B \rightarrow C =0$.
    \\ $K \rightarrow B \rightarrow F = 0$ because $K \rightarrow B$ is a kernel of $B \rightarrow F$:
    $$K \rightarrow B = Ker(B \rightarrow C) = Im(A \rightarrow B) = KerCok(A \rightarrow B) = Ker(A \rightarrow B).$$
    \item The third item implies the first:
    \\ Since $A \rightarrow B \rightarrow C =0$, $A \rightarrow B$ factors through $Ker(B \rightarrow C)$. \\ Therefore, by definition of image, $Im(A \rightarrow B) \subset Ker(B \rightarrow C)$.
    \\ On the other hand, since $K \rightarrow B \rightarrow F =0$, $K \rightarrow B$ factors through the kernel of $B \rightarrow F$:
    $$Ker(B \rightarrow C) = K \rightarrow B \subset Ker(B \rightarrow F) = KerCok(A \rightarrow B) = Im(A\rightarrow B).$$
\end{itemize}
\end{proof}

\begin{thm} \label{3.21}
Let $F: \mathcal{A} \rightarrow \mathcal{B}$ be an additive functor between abelian categories. The following are equivalent:
\begin{itemize}
    \item[(a)] F is an embedding.
    \item[(b)] F carries noncommutative diagrams into noncommutative diagrams.
    \item[(c)] F carries nonexact sequences into nonexact sequences.
\end{itemize}
\end{thm}

\begin{proof}
\begin{itemize}
    \item The first two statements are trivially equivalent.
    \item The third implies the first:
    Let $A' \xrightarrow{x} A \neq 0$. Then $A' \xrightarrow{1} A' \xrightarrow{x} A$ is not exact, so neither is $FA' \xrightarrow{1} FA' \xrightarrow{Fx} FA$, hence $Fx \neq 0$.
    \item The first implies the third:
    \\ Let $A' \rightarrow A \rightarrow A''$ be a nonexact sequence in $\mathcal{A}$. Let $0 \rightarrow K \rightarrow A \rightarrow A''$ and $A' \rightarrow A \rightarrow G \rightarrow 0$ be exact.
    By Lemma \ref{2.21}, either $A' \rightarrow A \rightarrow A'' \neq 0$ or $K \rightarrow A \rightarrow G \neq 0$. By assumption, $F$ applied to a nonzero map is nonzero, so we have two cases:
    \begin{enumerate}
        \item If $FA' \rightarrow FA \rightarrow FA'' \neq 0$ then by Lemma \ref{2.21} $FA' \rightarrow FA \rightarrow FA''$ is nonexact.
        \item If $FK \rightarrow FA \rightarrow FG \neq 0$, let $0 \rightarrow L \rightarrow FA \rightarrow FA''$ and $FA' \rightarrow FA \rightarrow H \rightarrow 0$ be exact in $\mathcal{B}$.
    \\ Since $FK \rightarrow FA \rightarrow FA'' = 0$, $FK \rightarrow FA$ factors through the kernel as $FK \rightarrow L \rightarrow FA$.
    \\ Since $FA' \rightarrow FA \rightarrow FG = 0$, $FA \rightarrow FG$ factors through the cokernel as $FA \rightarrow H \rightarrow FG$.
    \\ We see that $FA' \rightarrow FA \rightarrow FA''$ cannot be exact, otherwise Lemma \ref{2.21} would imply \\ $L \rightarrow FA \rightarrow H = 0$, then $$FK \rightarrow FA \rightarrow FG =FK \rightarrow L \rightarrow FA \rightarrow H \rightarrow FG= 0,$$ contradicting our assumption.
    \end{enumerate}
\end{itemize}
\end{proof}

\begin{cor} \label{exemb}
If a functor $F: \mathcal{A} \rightarrow \mathcal{B}$ between abelian categories is an exact embedding, then the exactness (resp. commutativity) of a diagram in $\mathcal{A}$ is equivalent to the exactness (resp. commutativity) of the $F$-image of the diagram.
\\
\end{cor}

\noindent Let us say an object $P$ in an abelian category $\mathcal{A}$ is \textbf{projective} if the functor $Hom(P,-): \mathcal{A} \rightarrow Ab$ is exact.
\\ Of course, $Hom(A,-)$ is left-exact for any $A \in \mathcal{A}$, so we may equally just demand right-exactness in this definition.
\\ Unpacking the definition, we see that $P$ is projective iff for any map $P \xrightarrow{p} B$ and epic $A \xrightarrow{e} B$, there is a map $P \xrightarrow{\tilde{p}} A$ (a \textit{lift} of $p$) such that $e \circ \tilde{p} = p$.
\\

\begin{prop} \label{3.32}
If $\{P_j\}$ is a family of projectives in an abelian category, then the direct sum $\Sigma_j P_j$ (if it exists) is projective.
\end{prop}

\begin{proof}
A map $\Sigma_j P_j \xrightarrow{p} B$ is given by individual maps $P_i \xrightarrow{p_j} B$. If we have an epic $A \xrightarrow{e} B$, there are componentwise lifts a map $P_j \xrightarrow{\tilde{p_j}} A$. That is, for each $i$, $e \circ \tilde{p_j} = p_j$. These collect into a map $\Sigma_j P_j \xrightarrow{\tilde{p}} A$ which lifts $p$: $e \circ \tilde{p} = p$, because these maps agree on each $P_j$:
$$(e \circ \tilde{p}) \circ \iota_j = e \circ (\tilde{p} \circ \iota_j)= e \circ (\tilde{p_j}) = p_j = p \circ \iota_j,$$
where $\iota_j$ is the $j$th inclusion into the sum.
\\
\end{proof}

\noindent Let us say an object $G \in \mathcal{A}$ is a \textbf{generator} if the functor $Hom(G,-): \mathcal{A} \rightarrow Ab$ is an embedding.
\\

\begin{prop} \label{3.33}
The following are equivalent:
\begin{itemize}
    \item $G$ is a generator.
    \item For every $A \rightarrow B \neq 0$ there is a map $G \rightarrow A$ such that $G \rightarrow A \rightarrow B \neq 0$.
    \item For every proper subobject of $A$ there is a map $G \rightarrow A$ whose image is not contained in the given subobject.
\end{itemize}
\end{prop}

\begin{proof}
\begin{itemize}
    \item Unpacking the definition, $G$ is a generator if and only if the function $$Hom(A,B) \rightarrow Hom(Hom(G,A),Hom(G,B)), f \mapsto f \circ -$$ is injective.
\\ This is if and only if for any nonzero $f \in Hom(A,B)$, the map $f \circ -$ is nonzero, meaning there is some $g \in Hom(G,A)$ with $f \circ g$ nonzero.
\\ Hence the first two statements are equivalent.

    \item The second statement implies the third.
    \\ Let $C \xrightarrow{s} A$ be a proper subobject. In particular $s$ is not epic, otherwise it would be an isomorphism as abelian categories are balanced. Take its cokernel $A \xrightarrow{c} B \neq 0$. There is some $G \xrightarrow{g} A$ with $cg \neq 0$. $Im(g)= I \xrightarrow{i} A$ is not contained in $C$. If it were, then by definition there would be a map $I \xrightarrow{f} C$ with $s \circ f = i $. Then
    $$cg = G \rightarrow I \xrightarrow{i} A \xrightarrow{c} B = G \rightarrow I \xrightarrow{f} C \xrightarrow{s} A \xrightarrow{c} B = 0,$$
    since $c$ was a cokernel of $s$. This is a contradiction.
    
    \item The third statement implies the second.
    \\ Given $A \xrightarrow{c} B \neq 0$, its kernel $K$ is a proper subobject of $A$, so there is some $G \xrightarrow{g} A$ whose image is not contained in $K$. In particular $cg \neq 0$.
    \\
\end{itemize}
\end{proof}

\begin{prop}
If $P$ is projective then it is a generator iff $Hom(P,A)$ is nontrivial for all nontrivial $A$.
\end{prop}

\begin{proof}
\begin{itemize}
    \item Let $P$ be a generator, and $A \neq 0$. Then $A \xrightarrow{1} A \neq 0$, and by the result above, there is some $P \xrightarrow{g} A$ with $g=1g \neq 0$.
    \item Let $P$ be projective, but not a generator. There is some $A \xrightarrow{c} B \neq 0$ such that for every $P \xrightarrow{g} A$, $cg=0$.
    \\ $c$ factors through $Im(c) = I \xrightarrow{i} B$ as $c= A \rightarrow I \xrightarrow{i} B$, where $A \rightarrow I$ is epic.
    \\ Then $I$ is nontrivial with trivial $Hom(P,I)$.
\end{itemize}
\end{proof}

\noindent Say a category is \textbf{well-powered} if the family of subobjects of any object is a set.

\begin{prop} \label{3.35}
An abelian category that has a generator is well-powered.
\end{prop}

\begin{proof}
Let $G$ be a generator, and $A$ any object. Then a subobject $A' \rightarrow A$ is distinguished by the subset $Hom(G,A') \subset Hom(G,A)$. (We have identified $Hom(G,A')$ with its image under $Hom(G,-)(A' \rightarrow A)$. In other words, there are no more subobjects of $A$ than subsets of $Hom(G,A)$.)
\end{proof}
%%%%%%%%%%%******

\begin{prop}\label{3.36}
$G$ is a generator in a cocomplete abelian category $\mathcal{A}$ iff for every $A \in \mathcal{A}$ the obvious map $\Sigma_{Hom(G,A)}G \rightarrow A$ is epic.
\end{prop}

\begin{proof} Let $G$ be a generator. Suppose for a contradiction there is some $A \in \mathcal{A}$ with $$A \rightarrow B := Cok(\Sigma_{Hom(G,A)}G \rightarrow A) \neq 0.$$
Then there is a map $G \rightarrow A$ with $G \rightarrow A \rightarrow B \neq 0$, but this contradicts that $$\Sigma_{Hom(G,A)}G \rightarrow A \rightarrow B = 0.$$
\\

\noindent Conversely, suppose $\Sigma_{Hom(G,A)}G \rightarrow A$ is epic, so its cokernel is zero. Suppose for a contradiction there is some $A \rightarrow B \neq 0$ such that every $G \rightarrow A$ has $G \rightarrow A \rightarrow B =0$. Then we have $\Sigma_{Hom(G,A)}G \rightarrow A \rightarrow B = 0$, so $A \rightarrow B$ factors through $Cok(\Sigma_{Hom(G,A)}G \rightarrow A)=0$. Then $A \rightarrow B =0$.
\end{proof}

\noindent The dual notions are as follows:
\\ An object $Q$ is \textbf{injective} if the functor $Hom(-,Q): \mathcal{A}^{op} \rightarrow Ab$ is exact.
\\ An object $C$ is a \textbf{cogenerator} if the functor $Hom(-,C): \mathcal{A}^{op} \rightarrow Ab$ is an embedding.
\\

\noindent Note that $Q$ is injective in $\mathcal{A}$ iff it is projective in $\mathcal{A}^{op}$, and $C$ is a cogenerator for $\mathcal{A}$ iff it is a generator for $\mathcal{A}^{op}$.

\begin{prop} \label{3.37}
Let $\mathcal{A}$ be a complete abelian category with a generator.
\\ There is, out of every object in $\mathcal{A}$, a monic to an injective object iff $\mathcal{A}$ has an injective cogenerator.
\end{prop}

\begin{proof}
\begin{itemize}
    \item Let $C$ be an injective cogenerator for $\mathcal{A}$, and $A \in \mathcal{A}$. The obvious map $A \rightarrow \Pi_{Hom(A,C)}C$ is monic, and $\Pi_{Hom(A,C)}C$ is injective. (We are using the duals of Propositions \ref{3.32} and \ref{3.36}.)
    \item Let $G$ be a generator for $\mathcal{A}$.
    \\ By Proposition \ref{3.35}, the class of quotient objects of $G$ is a set. (The class of subobjects of $G$ is a set, but this is in bijection with the class of quotient objects by Theorem {2.11}.)
    \\ Therefore we may define $P = \Pi_{\{ \text{quotient objects $Q$ of }G \} } Q$.
    \\
    
    By assumption we have a monic $P \rightarrow E$ where $E$ is injective. We claim $E$ is a cogenerator.
    \\Let $A \xrightarrow{c} B \neq 0$.
    \\ It is enough, by the dual of Proposition \ref{3.33}, to name some $B \rightarrow E $ such that $A \rightarrow B \rightarrow E \neq 0$.
    \\ Well, since $G$ is a generator, there is $G \xrightarrow{g} A$ with $cg \neq 0$.
    \\ Let $I \xrightarrow{i} B$ be the image of $G \xrightarrow{g} A \xrightarrow{c} B$, so $G \xrightarrow{cg} B = G \xrightarrow{cg} I \xrightarrow{i} B $.
    \\ Let $I \rightarrow P \rightarrow E$ be a monic $m$. (Since $I$ is a quotient object of $G$, it appears as a factor in $P$, so we may just take $I \rightarrow P$ to consist of the identity $I \rightarrow I$ and zero maps from $I$ to any other factor of $P$. This is monic because that identity component $I \rightarrow I$ is monic.)
    \\ Since $E$ is injective and $I \xrightarrow{i} B$ is monic, there is some $B \xrightarrow{b} E$ such that $bi = I \xrightarrow{m} E$.
    \\ We indeed have $A \xrightarrow{c} B \xrightarrow{b} E \neq 0$, since
    $$bcg = bicg = mcg \neq 0.$$
    (The last step is because $cg \neq 0$ and $m$ is monic.)
    \\
\end{itemize}
\end{proof}

\noindent Recall that a \textbf{subcategory} $\mathcal{A'}$ of the category $\mathcal{A}$ is just a subclass of the objects of $\mathcal{A}$, with, for any two objects $A', A$ in this subclass, a subclass of $Hom(A', A)$ closed under composition and identities. $\mathcal{A'}$ is, of course, a category, and there is an obvious inclusion functor $\mathcal{A}' \rightarrow \mathcal{A}$.
\\

\noindent Let $\mathcal{A'}$ be a subcategory of abelian category $\mathcal{A}$. We say $\mathcal{A}'$ is \textbf{exact} if $\mathcal{A'}$ is abelian and the inclusion functor is exact. The inclusion functor is automatically an embedding, so in this situation a diagram in $\mathcal{A}'$ is exact iff it is exact in $\mathcal{A}$ --- this was Corollary \ref{exemb}.
\\

\noindent Let us say that a functor $F: \mathcal{A} \rightarrow \mathcal{B}$ is \textbf{full} if for any $A_1, A_2 \in \mathcal{A}$ we have that the function $Hom(A_1,A_2) \rightarrow Hom(FA_1,FA_2)$ is surjective.
\\
\noindent A subcategory is \textbf{full} if the inclusion functor is full. A full subcategory of $\mathcal{A}$ can be specified simply by naming a subclass of the objects of $\mathcal{A}$.
\\

\noindent We also remark that any functor $F: \mathcal{A} \rightarrow \mathcal{B}$ restricts in the obvious way to a functor $F \ _{\mkern 1mu \vrule height 2ex\mkern2mu \mathcal{A'}}: \mathcal{A'} \rightarrow \mathcal{B}$ on any subcategory $\mathcal{A'}$ of $\mathcal{A}$.
\\ When $F$ is exact, full, or an embedding, then the restriction $F \ _{\mkern 1mu \vrule height 2ex\mkern2mu \mathcal{A'}}$ will respectively be exact, full, or an embedding.
\section{A special case of Freyd-Mitchell}

An abelian category $\mathcal{A}$ is \textbf{fully abelian} if for every full small exact subcategory $\mathcal{A'}$ of $\mathcal{A}$ there is a ring $R$ and a full exact embedding of $\mathcal{A}'$ into $R$-Mod.
\\

\noindent We shall now state a special case of the Freyd-Mitchell embedding theorem, that is easy to prove.

\begin{thm}[Mitchell] \label{Mitch}
A cocomplete abelian category with a projective generator is fully abelian.
\end{thm}

\begin{proof}
Let $\mathcal{A}'$ be a small full exact subcategory of a cocomplete category $\mathcal{A}$. Let $P'$ be a projective generator for $\mathcal{A}$. We wish to give a full exact embedding of $\mathcal{A}'$ into $R$-Mod, for some ring $R$.
\\

\noindent First of all, let us slightly modify $P'$.
\\ For each $A \in \mathcal{A}'$ consider the epic $\Sigma_{Hom(P',A)}P' \rightarrow A$ from Proposition \ref{3.36}.
\\ Let $I = \bigcup_{A \in \mathcal{A}'} Hom(P',A)$. Define $P=\Sigma_I P'$.
\\ By Proposition \ref{3.32}, $P$ is still a projective generator, but now we have an additional property: for each $A \in \mathcal{A}'$ there is an epic $P \rightarrow A$.
\\ (For instance, define $P \rightarrow A$ as the epic $\Sigma_{Hom(P',A)}P' \rightarrow A$ on the summands indexed over by $Hom(P',A)$, and as zero on all other summands.)
\\

\noindent Let $R$ be the ring $End(P)$ of endomorphisms on $P$.
\\ We had previously upgraded the functor $Hom(P,-): \mathcal{A} \rightarrow Set$ to a functor $Hom(P,-): \mathcal{A} \rightarrow Ab$, but now let us upgrade it further to a functor $Hom(P,-): \mathcal{A} \rightarrow R$-Mod.
\begin{itemize}
    \item For every $A \in \mathcal{A}$, the abelian group $Hom(P,A)$ has a canonical $R$-module structure:
    \\ given $P \xrightarrow{x} A \in Hom(P,A)$ and $P \xrightarrow{r} P \in R$, define $$r \cdot x = x \circ r \in Hom(P,A).$$
    \item For every map $A \xrightarrow{y} B$ in $\mathcal{A}$, the induced map $Hom(P,A) \xrightarrow{y \circ -} Hom(P,B)$ is $R$-linear:
    $$(y \circ -)(r \cdot x) = y \circ (r \cdot x) = y \circ (x \circ r) = (y \circ x) \circ r = r \cdot(y \circ x) = r \cdot ((y \circ -) x).$$
\end{itemize}
Hence we do get a functor $F=Hom(P,-): \mathcal{A} \rightarrow R$-Mod.
\\ $F$ is an exact embedding since $P$ is a projective generator.
\\ (To be slightly pedantic, the functor $Hom(P,-): \mathcal{A} \rightarrow Ab$ is an exact embedding by definition of $P$ as a projective generator, but $R$-Mod is an exact subcategory of $Ab$ --- the forgetful inclusion functor $R\text{-Mod} \rightarrow Ab$ has left and right adjoints, so it preserves finite limits and colimits, so it is exact.)
\\ The restriction $F \ _{\mkern 1mu \vrule height 2ex\mkern2mu \mathcal{A'}}$ is therefore an exact embedding; it only remains to show it is full.
\\

\noindent Suppose we have $A, B \in \mathcal{A'}$ and a map $FA \xrightarrow{\tilde{y}} FB$ in $R$-Mod. We must exhibit a map $A \xrightarrow{y} B$ in $\mathcal{A'}$ such that $Fy = \tilde{y}$, where $Fy = y \circ -$.
\\ Since $A,B \in \mathcal{A}'$, we have exact sequences $0 \rightarrow K \rightarrow P \rightarrow A \rightarrow 0$ and $P \rightarrow B \rightarrow 0$ in $\mathcal{A}$ coming from the epics $P \rightarrow A$ and $P \rightarrow B$. (Just take $K \rightarrow P = Ker(P \rightarrow A)$.)
\\ Since $FP=R$, taking $F$ gives us the following commutative diagram in $R$-Mod:

\[ \begin{tikzcd}
0 \arrow{r} &FK \arrow{r} &R \arrow{d}{f} \arrow{r} &FA \arrow{d}{\tilde{y}} \arrow{r} &0 \\
& &R  \arrow{r} &FB \arrow{r} &0
\end{tikzcd} 
\]
where $f$ is a lift of $R \rightarrow FA \xrightarrow{\tilde{y}} FB$. ($R$ is projective and $R \rightarrow FB$ is epic.)
\\

\noindent Since $R$ is a ring, we have $End(R) \cong R^{op}$ --- in other words, any endomorphism on $R$ is given by multiplication \textit{on the right} by some $R$-element. Hence, write $f(s)=sr = s \circ r$ for all $s \in R$, where $P \xrightarrow{r} P \in R$.
\\
%%%%%%%%%%%%%%%%%%%%%%%%
\noindent Return to $\mathcal{A}$: in the diagram
\[ \begin{tikzcd}
0 \arrow{r} &K \arrow{r} &P \arrow{d}{r} \arrow{r} &A \arrow{r} &0 \\
& &P  \arrow{r} &B \arrow{r} &0
\end{tikzcd} 
\]
we have that $K \rightarrow P \xrightarrow{r} P \rightarrow B = 0$, as $FK \rightarrow R \xrightarrow{f} R \rightarrow FB = 0$ and $F$ is an embedding.
\\ Hence $P \xrightarrow{r} P \rightarrow B$ factors through the cokernel $P \rightarrow A$ --- there is $A \xrightarrow{y} B$ such that

\[ \begin{tikzcd}
P \arrow{r} \arrow{d}{r} & A \arrow{d}{y} \\
P \arrow{r} & B & \text{commutes.}
\end{tikzcd} 
\]
Hence

\[ \begin{tikzcd}
R \arrow{r} \arrow{d}{f} & F A \arrow{d}{Fy} \\
R \arrow{r} & FB  & \text{commutes.}
\end{tikzcd} 
\]
\\

\noindent Therefore $R \rightarrow FA \xrightarrow{Fy} FB = R \xrightarrow{f} R \rightarrow FB = R \rightarrow FA \xrightarrow{\tilde{y}} FB$.
\\ Since $R \rightarrow FA$ is epic, $Fy = \tilde{y}$.
\\
\\
\end{proof}

\noindent The full statement of the \textbf{Freyd-Mitchell embedding theorem} is: \textit{Every abelian category is fully abelian.}
\\ We have just shown that this is true if our category is cocomplete with a projective generator.
\\ Therefore, if we want to show that every abelian category is fully abelian, it is enough to solve the following problem: Given a small abelian category $\mathcal{A}$, find a cocomplete abelian category $\mathcal{L}$ with a projective generator and an exact full embedding $\mathcal{A} \rightarrow \mathcal{L}$. (The composition of two full exact embeddings is again a full exact embedding.)

\section{Functor categories}

Let $\mathcal{A}$ be a small abelian category. Let $[\mathcal{A}, Ab]$ denote the category of additive functors from $\mathcal{A}$ to $Ab$. Its objects are functors, and its maps are natural transformations.
\\
\begin{thm}
$[\mathcal{A}, Ab]$ is an abelian category.
\end{thm}

\begin{proof}
We briefly run through the axioms.
\begin{itemize}
    \item[A0.] The constantly zero functor is a zero object.
    \item[A1, A1*.] (Binary) sums and products are computed pointwise. Given $F_1, F_2 \in [\mathcal{A}, Ab]$, define a functor $F_1 \oplus F_2$ on objects as $( F_1 \oplus F_2 )(A) = F_1(A) \oplus F_2(A)$ and on maps as
    $$( F_1 \oplus F_2 )(x) = \begin{pmatrix}
    F_1(x) & 0  \\
    0 & F_2(x) 
\end{pmatrix}. $$
This plays the role of binary sum and product.

    \item[A2.] Let $F_1 \rightarrow F_2$ in $[\mathcal{A}, Ab]$. We construct a kernel $K \rightarrow F_1$.
    \\ For each $A \in \mathcal{A}$, let $K(A):=Ker(F_1A \rightarrow F_2A)$.
    \\ Given $A \xrightarrow{x} B$ in $\mathcal{A}$ there is a unique map $K(x): K(A) \rightarrow K(B)$ such that
    
\[ \begin{tikzcd}
K(A) \arrow{r} \arrow{d}{K(x)} & F_1(A) \arrow{d}{F_1(x)} \\
K(B) \arrow{r} & F_1(B) & \text{commutes.}
\end{tikzcd} 
\]

The uniqueness forces $K$ to be a functor, and $K \rightarrow F_1$ is a natural transformation. (The diagram above is a naturality square.)
    
    \item[A2*.] Dually to A2, we construct a cokernel $F_2 \rightarrow C$ for each $F_1 \rightarrow F_2$ pointwise.
    \item[A3.] The construction in A2 shows that a natural transformation $F_1 \rightarrow F_2$ is monic in $[\mathcal{A}, Ab]$ iff $F_1A \rightarrow F_2A$ is monic in $\mathcal{A}$ for each $A$. The construction for A2* shows that if $F_1 \rightarrow F_2$ is monic, then it is a kernel of its cokernel.
    %%%%%%%%%%%%%%%%%%%%%%%%%%
    \item[A3*.] Dual to A3.
\end{itemize}
\end{proof}

\noindent These constructions indicate that a sequence $F' \rightarrow F \rightarrow F''$ is exact in $[\mathcal{A}, Ab]$ iff the sequences \\ $F'A \rightarrow FA \rightarrow F''A$ are exact in $\mathcal{A}$ for all $A \in \mathcal{A}$.
\\ More formally, the \textbf{evaluation functor} $E_A: [\mathcal{A},Ab] \rightarrow Ab$ defined by $E_A(F_1 \xrightarrow{\eta} F_2) = F_1A \xrightarrow{\eta(A)} F_2A$ is an exact functor for each $A \in \mathcal{A}$.
\\ The product $(\Pi_{\mathcal{A}}E_A): [\mathcal{A}, Ab] \rightarrow Ab$ defined by $(\Pi_{\mathcal{A}}E_A)(F) = \Pi_{\mathcal{A}}E_A(F) = \Pi_{\mathcal{A}}FA$ is an exact embedding.
%%%%%%%%%%%%%%
\\
\begin{prop}
$[\mathcal{A}, Ab]$ is a bicomplete abelian category.
\end{prop}

\begin{proof}
Let $\{F_i\}_I$ be a (small) collection of functors in $[\mathcal{A}, Ab]$.
\\ We construct $\Pi_IF_i$ and $\Sigma_IF_i$ pointwise, just as we did finite direct sums:
$$(\Pi_IF_i)(A)=\Pi_IF_iA \qquad \text{and} \qquad (\Sigma_IF_i)(A)=\Sigma_IF_iA.$$
\end{proof}

\noindent The next definition generalises a property that is possessed by categories like $Ab$ and $R$-Mod, where $R$ is a ring.
\begin{defn}
Let $\mathcal{A}$ be a bicomplete well-powered abelian category.
\\ We say $\mathcal{A}$ is a \textbf{Grothendieck} category if for each chain $\{S_i\}_I$ in the lattice of subobjects of an object $S$, and $T$ is any subobject of $S$, then we have $$T \cap \bigcup S_i = \bigcup (T \cap S_i).$$
\end{defn}

\noindent That $R$-Mod satisfies this property really is quite trivial, because the union and the intersection are just set-theoretic union and intersection. It was important, then, that we demanded the family of subobjects to be a chain --- this guarantees that the set-theoretic union is again a module.
\\
\begin{prop}
$[\mathcal{A}, Ab]$ is a Grothendieck category.
\end{prop}

\begin{proof}
$[\mathcal{A}, Ab]$ is certainly well-powered (Proposition \ref{3.35} and Theorem \ref{5.35}, for instance).
\\ Note that given a collection $\{F_i\}_I$ of subfunctors of $F$, their union and intersection are constructed pointwise:
$$(\bigcup F_i)(A) = \bigcup (F_iA) \subset FA,$$
since we know that $F_i \rightarrow F$ is monic only if each component is monic.
\\ Hence, given a chain $\{F_i\}$ and subfunctor $H \subset F$, we have
$$(H \cap \bigcup F_i)(A) = HA \cap \bigcup F_iA = \bigcup (HA \cap F_iA) = (\bigcup (H \cap F_i))(A),$$
where the second equality uses that $Ab$ is Grothendieck.
\\
\end{proof}

\noindent Recall that the \textbf{(co)-Yoneda embedding} is the functor $H: \mathcal{A}^{op} \rightarrow [\mathcal{A}, Ab]$ given on objects by \\ $H(A) = Hom(A, -)$, and on maps by $H(A \xrightarrow{x} B) = Hom(B, -) \xrightarrow{ (x , -)} Hom(A,-)$.
\\ If we denote $Hom(A,-)$ by $H^A$, then we may as well denote $H(A \xrightarrow{x} B)$ by $H^B \xrightarrow{H^x} H^A$.
\\

\begin{thm}
The Yoneda embedding $H: \mathcal{A}^{op} \rightarrow [\mathcal{A}, Ab]$ is left-exact.
\end{thm}

\begin{proof}
Let $0 \rightarrow A'\rightarrow A \rightarrow A''$ be exact in $\mathcal{A}$. We show $H^{A''} \rightarrow H^A \rightarrow H^{A'} \rightarrow H^0$ is exact in $[\mathcal{A}, Ab]$.
\\ We know this holds iff $Hom(0,B) \rightarrow Hom(A',B) \rightarrow Hom(A,B) \rightarrow Hom(A'',B)$ is exact in $Ab$ for each $B \in \mathcal{A}$, but this holds because $Hom(-,B): \mathcal{A}^{op} \rightarrow Ab$ is left exact.
\\
\end{proof}

\noindent We recall the following famous lemma from category theory.

\begin{lem}[Yoneda Lemma]
$Hom(H^A,F)$ is naturally isomorphic to $F(A)$ in $A \in \mathcal{A}^{op}$ and $F \in [\mathcal{A}, Ab]$.
\\
\end{lem}

\begin{thm} \label{5.35}
$\Sigma_\mathcal{A} H^A$ is a projective generator for $[\mathcal{A}, Ab]$.
\end{thm}

\begin{proof}
Let us be more specific about what the Yoneda Lemma says.
\\ There are functors $D, E : \mathcal{A}^{op} \times [\mathcal{A}, Ab] \rightarrow Ab$ defined by
$$D = \mathcal{A}^{op} \times [\mathcal{A}, Ab] \xrightarrow{H \times 1} [\mathcal{A}, Ab] \times [\mathcal{A}, Ab] \xrightarrow{Hom_{[\mathcal{A}, Ab]}} Ab,$$
(so $D(A,F) = Hom (H^A, F)$,)
\\ and the \textit{evaluating functor} $$E(A,F) = F(A); \quad E(A,F_1 \xrightarrow{\eta} F_2) = F_1(A) \xrightarrow{\eta_A} F_2(A), \quad E(A_1 \xrightarrow{x} A_2, F) = F(A_1) \xrightarrow{F(x)} F(A_2).$$
The Yoneda Lemma says $D$ is naturally isomorphic to $E$.
\\ Hence, as functors $[\mathcal{A}, Ab] \rightarrow Ab$, we have that
$Hom(\Sigma_\mathcal{A} H^A, - )$ is naturally isomorphic to $(\Pi E_A)$:
$$Hom(\Sigma_\mathcal{A} H^A, - ) = \Pi_\mathcal{A} Hom( H^A, - ) =\Pi_\mathcal{A} D( A, -) \cong \Pi_\mathcal{A} E( A, -)  = (\Pi_\mathcal{A} E_A), $$
The latter is an exact embedding.
\end{proof}

\begin{thm} \label{5.36}
The Yoneda embedding $H: \mathcal{A}^{op} \rightarrow [\mathcal{A}, Ab]$ is a full embedding.
\end{thm}

\begin{proof}
This follows immediately from setting $F = H^B$ in the Yoneda Lemma: $$Hom_{[\mathcal{A},Ab]}(H^A,H^B) \cong H^B(A)=Hom_\mathcal{A}(B,A)=Hom_{\mathcal{A}^{op}}(A,B).$$
\end{proof}
\section{Injective Envelopes}

The key result of this section will be the following: In a Grothendieck category that has a generator, every object has an injective envelope.
\\

\noindent In particular this applies to $[\mathcal{A}, Ab]$, and will be very useful in the next section.
\\

\noindent Throughout let $\mathcal{A}$ be an abelian category. Given an object $A \in \mathcal{A}$, an \textbf{extension} of $A$ is simply a monic $A \rightarrow B$ out of $A$. Sometimes we will call $B$ itself an extension of $A$.
\\ A \textbf{trivial} extension of $A$ is a \textbf{split} monic --- a monic $A \xrightarrow{x} B$ for which there is some $B \xrightarrow{y} A$ with $yx=1_A$. Equivalently, $a \rightarrow B$ is a trivial extension if there is an object $C$ with $B = A \oplus C$, and $A \rightarrow B$ is the inclusion $A \xrightarrow{i_1} A \oplus C$. ($C$ must then be the cokernel of $A \rightarrow B$.)

\begin{prop}
An object in $\mathcal{A}$ is injective iff it has only trivial extensions.
\end{prop}

\begin{proof}
The forward direction is clear: if $I$ is injective and $I \xrightarrow{x} B$ is monic, then $I\xrightarrow{1}I$ extends to a map $B \xrightarrow{y} I$, meaning $yx = 1_I$.
\\

\noindent For the reverse direction, suppose $E$ has only trivial extensions. Let $A \xrightarrow{x} B$ be monic, and $A \xrightarrow{a}E$ be any map. We find a map $B \xrightarrow{y} E$ with $yx=a$.
\\ Make a pushout diagram
\[ \begin{tikzcd}
A \arrow{r}{x} \arrow[swap]{d}{a} & B \arrow{d}{b} \\
E \arrow{r}{e} & P
\end{tikzcd}
\]
and observe that since $x$ is monic, so is $e$. By assumption, $P$ must be a trivial extension of $E$, meaning there is $P \xrightarrow{f} E$ with $fe=1_E$.
Put $y=B \xrightarrow{fb} E$; then $yx=fbx=fea=1_Ea=a$.
\\
\end{proof}

\noindent An \textbf{essential extension} is a monic $A \rightarrow B$ such that for every nonzero monic $B' \rightarrow B$, ther intersections (of the images) of $A \rightarrow B$ and $B' \rightarrow B$ are nonzero.
\\

\begin{prop}
An extension $A \rightarrow B$ is essential if for every $B \rightarrow F$ such that $A \rightarrow B \rightarrow F$ is monic, we have that $B \rightarrow F$ is monic.
\end{prop}

\iffalse
Note that we never have an essential extension $0 \rightarrow B$ if $B \neq 0$. Indeed, the intersection of $0 \rightarrow B$ and $B \xrightarrow{1} B$ is zero.
\fi
\begin{proof}
\begin{itemize}
    \item Let $A \rightarrow B$ be essential, and $B \rightarrow F$ be such that $A \rightarrow B \rightarrow F$ is monic. We claim $B \rightarrow F$ is monic. Suppose not, then $B' \rightarrow B := Ker (B \rightarrow F) \neq 0$ is monic, so by assumption $$(A \rightarrow B) \cap (B' \rightarrow B) \neq 0.$$
    On the other hand, we show the intersection is zero, for a contradiction. Suppose the monic $C \rightarrow B$ is contained in the intersection, so $C \rightarrow B$ factors as $C \rightarrow A \rightarrow B$, and also factors through the kernel of $B \rightarrow F$. In particular, $$C \rightarrow A \rightarrow B \rightarrow F=C \rightarrow B\rightarrow F= 0.$$
    Since $C \rightarrow A$ is monic, we conclude $A \rightarrow B \rightarrow F = 0$, but this was a monic, so $A = 0$. Then $A \rightarrow B = 0$, so the intersection has to be zero.
    \item Conversely, suppose $B' \rightarrow B$ is a nonzero monic with $ (A \rightarrow B) \cap (B' \rightarrow B) = 0$. Set $B \rightarrow F := Cok(B' \rightarrow B)$. \\ We see that $B \rightarrow F$ is not monic --- otherwise $0 = Ker(B \rightarrow F) = B' \rightarrow B$. \\ On the other hand, $Ker(A \rightarrow B\rightarrow F) = 0$:
    Suppose $A' \rightarrow A$ is such that $A' \rightarrow A \rightarrow B \rightarrow F = 0$. We must show it is zero.
    \\ Consider the monic $Im(A' \rightarrow A \rightarrow B) = I \rightarrow B$. By definition of image, this factors through $A \rightarrow B$. It also factors through $B' \rightarrow B = Ker(B \rightarrow F)$, since $A' \rightarrow I$ is epic and
    $$A' \rightarrow I \rightarrow B \rightarrow F = A' \rightarrow A \rightarrow B \rightarrow F = 0.$$
    Hence it lies in the intersection $(A \rightarrow B) \cap (B' \rightarrow B) = 0$ as required.
\end{itemize}
\end{proof}

\begin{thm} \label{6.13}
In a Grothendieck category, an object is injective iff it has no proper essential extensions.
\end{thm}

\begin{proof}
Certainly if $E$ is injective, then its only proper extensions are trivial, $E \xrightarrow{i_1} E \oplus B$, $B \neq 0$. 
\\ Then $E \oplus B \xrightarrow{\pi_1} E$ is not monic (it is epic but not an isomorphism); however $\pi_1 i_1=1_E$ is monic. By definition this is not essential.
\\

\noindent Conversely, let $E$ have no proper essential extensions. Let $E \rightarrow B$ be any extension; we show it must be trivial.
\\ Let $\mathcal{F}$ be the poset (ordered by inclusion) of subobjects of $B$ which have zero intersections with (the image of) $E \rightarrow B$.
\begin{itemize}
    \item[\textbf{Claim:}] If $\{B_i\}_I$ is an ascending chain in $\mathcal{F}$ then $\bigcup B_i \in \mathcal{F}.$
    \item[\textbf{Proof:}] $\bigcup B_i$ exists as a subobject of $B$. We show it has zero intersection with $Im(E \rightarrow B) = I \rightarrow B$:
    $$I \cap \bigcup B_i = \bigcup (I \cap B_i) = \bigcup 0 = 0.$$
    The claim is proven.
\end{itemize}
\noindent Hence, Zorn's Lemma guarantees us a maximal element $B' \subset B$ of $\mathcal{F}$.
\\ Let us switch perspectives by taking cokernels, to get a corresponding family $\tilde{\mathcal{F}}$ of quotient objects of $B$, where
$$B \rightarrow F \in \tilde{\mathcal{F}} \text{ iff } E \rightarrow B \rightarrow \text{ F is monic.}$$
This must have a minimal element $B \rightarrow B''$ (corresponding to $B' \subset B$).
%%%%%%%%%%%%%%%%%%%%%%%%
\\ Certainly $E \rightarrow B \rightarrow B''$ is monic; let us show it is essential.
\\ Suppose $B'' \rightarrow F$ is such that $E \rightarrow B \rightarrow B'' \rightarrow F$ is monic, then by definition, the coimage of $B \rightarrow B'' \rightarrow F$ is smaller than $B \rightarrow B''$. By minimality of $B''$, it must be equal to this coimage, and in particular is monic.
\\

\noindent By hypothesis, the essential extension $E \rightarrow B \rightarrow B''$ cannot be proper, so it is an isomorphism, and $E \rightarrow B$ is trivial. (Writing $\psi = E \xrightarrow{j} B \xrightarrow{k} B''$, $\psi^{-1}k:B \rightarrow E$ is such that $\psi^{-1}k j = (kj)^{-1}kj=1_E$.)
\\
\end{proof}

\noindent The following falls out easily as a corollary. We include it because the key result of this section is proven similarly.
\begin{thm}[Baer's Criterion]
Let $R$ be a ring, and $A$ be a left $R$-module.
\\ If for every left ideal $I \subset R$ we have that $Hom(R,A) \rightarrow Hom(I,A)$ is epic, then $A$ is injective in $R$-Mod.
\end{thm}

\begin{proof}
By the theorem above, it suffices to show that $A$ has no proper essential extensions.
\\ Let $A \subset B$, $x \in B \backslash A$. We show $A \subset B$ is not essential.
\\ Let $R \xrightarrow{x} B$ be the map sending $1 \mapsto x$. Make a pullback diagram:

\[ \begin{tikzcd}
I \arrow{r}{i_1} \arrow[swap]{d}{i_2} & R \arrow{d}{x} \\
A \arrow{r}{j} & B
\end{tikzcd}
\]
$I = \{(a,r): a=rx\} = \{(rx,x): rx \in A\}$ may be identified with the ideal $\{r \in R: rx \in A\}$,
\\ so by assumption $I \xrightarrow{i_2} A$ extends to a map $R \rightarrow A$: there is some $y \in A$ with $I \xrightarrow{i_1} R \xrightarrow{y} A = I \xrightarrow{i_2} A$.
\\

\noindent We have $x-y \neq 0$ since $x \notin A \ni y$. On the other hand, the submodule $M = \{r(x-y): r \in R\}$ of $B$ generated by $x-y$ meets $A$ only trivially.
\\ In other words, consider the nonzero monic $M \subset B$. $B$ is not essential, because the intersection of the images of $A \subset B$ and $M \subset B$ is zero --- given $r(x-y) \in A$ where $r \in R$, then $rx=r(x-y)+ry \in A$, so $r \in I$, so $r(x-y) = 0$ because
$$rx= xi_1(r)= ji_2(r)=jyi_1(r) = y(r) = ry.$$
\\
\end{proof}

\begin{defn}
An \textbf{injective envelope} of $A$ is an injective essential extension.
\end{defn}

\noindent An injective envelope is a maximal essential extension and a minimal injection extension.
\\
%%%%%%%%%%%%%%%%%%%%%%%%%%%%%%
\begin{lem}
An essential extension of an essential extension is essential.
\end{lem}

\begin{proof}
Let $A \xrightarrow{a}B$, $B \xrightarrow{b}C$ be essential extensions. We show the extension $A \xrightarrow{ba} C$ is essential.
\\ Let $C \xrightarrow{c} F$ be such that $cba$ is monic.
\\ Since $a$ is essential, $cb$ is monic. Since $b$ is essential, $c$ is monic.
\\
\end{proof}

\begin{lem} \label{6.22}
Let $A \rightarrow E$ be an extension of $A$ in a Grothendieck category, and $\{E_i\}$ an ascending chain of subobjects between (the image of) $A$ and $E$. If $E_i$ is an essential extension of $A$ for each $i$, then $\bigcup E_i$ is an essential extension of $A$.
\end{lem}

\begin{proof}
Let $S$ be any nonzero subobject of $\bigcup E_i$.
\\ Then $S = S \cap \bigcup E_i = \bigcup (S \cap E_i)$, hence $S \cap E_i \neq 0$ for some $i$.
\\ Since $E_i$ is essential, we have $S \cap A = (S \cap E_i) \cap A \neq 0$.
\\
\end{proof}

\noindent Although $E$ does not appear explicitly in the proof above, the proof really does hinge on the fact that $A$ and the $E_i$ are contained in $E$; otherwise we could not even speak of $\bigcup E_i$.

\noindent It is the next lemma that asserts that every ascending chain of extensions may indeed be embedded in a common extension $E$, and therefore, the lemma above becomes the statement that every ascending chain of essential extensions is bounded by an essential extension.
\\

\begin{thm} \label{6.23}
Let $\mathcal{B}$ be a Grothendieck category, $J$ an ordered set, and $\{E_j \rightarrow E_k\}_{j < k}$ a family of monics such that whenever $j < k< l$, $E_j \rightarrow E_k \rightarrow E_l = E_j \rightarrow E_l$.
\\ Then there is an object $E \in \mathcal{B}$ such that whenever $j < k$, $$E_j \rightarrow E_k \rightarrow E = E_j \rightarrow E$$.
\end{thm}

\begin{proof}
Let $S = \Sigma_J E_j$. For each $j \in J$ let $E_j \xrightarrow{\iota_j} S$ be the $j$th inclusion. For each $j \in J$, define a map $S \xrightarrow{h_j} S$ on the component $E_k$ as

$$E_k \xrightarrow{\iota_k} S \xrightarrow{h_j} S = 
\left\{
	\begin{array}{ll}
		E_k \rightarrow E_j \xrightarrow{\iota_j} S  & \mbox{if } k \leq j \\
		E_k \xrightarrow{\iota_k} S & \mbox{if } j \leq k.
	\end{array}
\right.
$$

\noindent Let $S \xrightarrow{h} E$ be an epic such that $Ker(h) = \bigcup_k Ker(h_k)$. (Just take the cokernel of the subobject $\bigcup_k Ker(h_k)$ of $S$.)
\\ Note that $\{Ker(h_k)\}$ is an ascending family, because for $k \leq {k'}$ we have
$$S \xrightarrow{h_{k'}} S = S \xrightarrow{h_k} S \xrightarrow{h_{k'}} S.$$
%%%%%%%%%%%%%%%
\\ It remains to see that each $E_j \xrightarrow{\iota_j} S \xrightarrow{h} E$ is monic. \\ For this, it suffices to show that $Im(E_j \rightarrow S) \cap ( \bigcup_k Ker(h_k)) =0$.
%%%%%%%%%%%%%%%
\\ We know each $E_j \rightarrow S \xrightarrow{h_k} S$ is monic, so $Im(E_j \rightarrow S) \cap ( Ker(h_k)) =0$ for each $k$, and we are done by the Grothendieck axiom.
\\
\end{proof}

\noindent Recall our goal for this section: to prove that in a Grothendieck category with a generator, every object has an injective envelope.
\\ Let $\mathcal{B}$ be a Grothendieck category. By Theorem \ref{6.13} we may choose for each non-injective object $A \in \mathcal{B}$ a proper essential extension $E(A):= (A \rightarrow B)$. If $A \in \mathcal{B}$ is injective, setting $E(A) := A \rightarrow A$ already gives us an injective envelope.
\\Define $E^\gamma (A)$ by transfinite recursion, as follows:
\begin{itemize}
    \item on zero: $E^0 (A)=E(A)$;
    \item on successor ordinals: $E^{\gamma + 1} (A) = E \rightarrow E^\gamma (A) \rightarrow E(E^{\gamma} (A))$;
    \item on limit ordinals: $E^\alpha(A)$ is a minimal essential extension that bounds $E^\gamma (A)$ for all $\gamma < \alpha$. (Such an extension exists by Theorem \ref{6.23}.)
\end{itemize}
\noindent Then the sequence $\{ E^\gamma (A) \}$ becomes stationary precisely when it reaches an injective essential extension, i.e., an injective envelope of $A$. We show that this does happen, when $\mathcal{B}$ has a generator.

\begin{thm}
If $\mathcal{B}$ is a Grothendieck category with a generator $G$ then every object has an injective envelope.
\end{thm}

\begin{proof}
We start out similarly to the proof of Theorem \ref{Mitch}: Let $R = End (G)$; there is a functor $F: \mathcal{F} \rightarrow R$-Mod sending $\mathcal{B} \ni B \mapsto Hom(G,B) \in R$-Mod.
\begin{itemize}
    \item[\textbf{Claim}:] If $A \rightarrow E$ is an essential extension in $\mathcal{B}$, then $FA \rightarrow FE$ is an essential extension in $R$-Mod.
    \item[\textbf{Proof:}] $FA \rightarrow FE$ is an extension because $F=Hom(G,-)$ is left-exact.
    \\ Let $M \subset FE$ be a nonzero submodule, so there is $x \in M$. We need to construct a nonzero element in $M \cap Im(FA \rightarrow FE)$.
    \\ $x \in M \subset FE = Hom(G,E)$, so take a pullback diagram:
\[ \begin{tikzcd}
P \arrow{r} \arrow[swap]{d} & G \arrow{d}{x} \\
A \arrow{r} & E
\end{tikzcd}
\]
    $A \rightarrow E$ was essential and $x \neq 0$, so $P \neq 0$, and $G \xrightarrow{1} G$ factors as $G \rightarrow P \rightarrow G$.
    \\ Now $0 \neq G \rightarrow P \rightarrow G \xrightarrow{x} E$ is an element of $M$, and is contained in $Im(F(A \rightarrow E))$.
    %%%%%%%%%%%%%%%%%%%
\end{itemize}
\noindent Now, we use the fact that for any ring $R$, $R$-Mod \textit{has enough injectives}: there is an injective extension out of every $R$-module. In particular there is an injective extension $FA \rightarrow Q$, which factors by injectivity of $Q$ as $FA \rightarrow FE \rightarrow Q$. Further, we have that $FE$ is isomorphic to a subobject of $Q$.
\\ The above holds for \textit{any} essential extension $E$ of $A$, so, simply take any ordinal $\Omega$ whose cardinality is larger than that of the set of subobjects of $Q$. Since $F$ is an embedding, any sequence of proper essential extensions of $A$ must terminate before $\Omega$. (There are no more essential extensions $A \rightarrow E$ than the extensions $FA \rightarrow FE$, but there are no more of \textit{these} than subobjects of $Q$.)
%%%%%%%%%%%%%%%
\end{proof}
\section{The Embedding Theorem}

\begin{prop} \label{7.11}
If an object $E \in [\mathcal{A},Ab]$ is injective, then it is a right-exact functor.
\end{prop}

\begin{proof}
Let $A' \rightarrow A \rightarrow A'' \rightarrow 0$ be an exact sequence in $\mathcal{A}$. Applying the Yoneda embedding $H$, we obtain in $[\mathcal{A},Ab]$ an exact sequence
$$0 \rightarrow H^{A''} \rightarrow H^{A} \rightarrow H^{A'}.$$
By definition of $E$ being injective, the functor $Hom(-,E)$ is exact. Therefore we obtain in $Ab$ an exact sequence
$$ Hom(H^{A'},E) \rightarrow Hom(H^{A},E) \rightarrow Hom(H^{A''},E) \rightarrow 0.$$
By the Yoneda Lemma, this sequence is isomorphic to
$$ E(A') \rightarrow E(A) \rightarrow E(A'') \rightarrow 0,$$
so $E$ is right-exact.
\\
\end{proof}

\noindent A functor is \textbf{mono} if it preserves monics. In particular a right-exact functor is exact iff it is mono, so an injective mono functor is exact.
\\ The injective envelope of a mono functor is an exact functor:
\\

\begin{lem} \label{7.12}
Let $M \rightarrow E$ be an essential extension in $[\mathcal{A}, Ab]$. If $M$ is a mono functor, then so is $E$.
\end{lem}

\begin{proof}
Suppose $E$ is not mono, so there is a monic $A' \rightarrow A$ in $\mathcal{A}$ such that $EA' \rightarrow EA$ is not monic in $Ab$. There is $0 \neq x \in EA'$ with $(EA' \rightarrow EA)(x) =0$; we construct the subfunctor $F \subset E$ \textit{generated by} $x$ as follows.
\\ Define it on objects as $F(B) = \{ y \in EB: \text{ there exists } A' \rightarrow B \text{ in } \mathcal{A} \text{ such that } (EA' \rightarrow EB)(x) = y\}$, from which it follows that $$(EB' \rightarrow EB)(FB') \subset FB$$
for $B' \rightarrow B$: If $y \in FB'$ then there is $A' \rightarrow B'$ in $\mathcal{A}$ with $(EA' \rightarrow EB')(x)=y$. Then $A' \rightarrow B' \rightarrow B$ witnesses that $(EB' \rightarrow EB)(y) \in FB$.
\\ Hence we may define $F(B' \rightarrow B)$ by restriction:
$$F(B' \rightarrow B) = FB' \rightarrow FB, y \mapsto (EB' \rightarrow EB)(y).$$
(Functoriality is then tautological.)
\\ $F$ is still a set-valued functor. We would like to upgrade it to a functor $\mathcal{A} \rightarrow Ab$, and we do this by observing that $FB$ is a subgroup of $EB$:
\begin{itemize}
    \item $0 \in FB$, since the zero map sends $x$ to it.
    \item if $y,z \in FB$ then there are $f,g: A' \rightarrow B$ with $(Ef)(x)=y, (Eg)(x)=z$.
    \\ Then $(E(f-g))(x)=(Ef-Eg)(x)=(Ef)(x)-(Eg)(x)=y-z$, so $y-z \in FB$,
    \\ where the first equality uses Proposition \ref{7.11} and Corollary \ref{3.12}.
\end{itemize}

\noindent Since $x \in FA' \subset EA'$, we have $F \neq 0$. Since $M \rightarrow E$ is essential, $F \cap M \neq 0$, so there is some $B$ with $FB \cap MB \neq 0$, so there is $0 \neq y \in FB \cap MB$.
\\ Since $y \in FB$, there is $A' \rightarrow B$ with $y=(EA' \rightarrow EB)(x)$. Let
\[ \begin{tikzcd}
A' \arrow{d} \arrow{r} &A \arrow{d}\\
B  \arrow{r} &P
\end{tikzcd} 
\]
be a pushout diagram. Since $A' \rightarrow A$ was monic, so will be $B \rightarrow P$. Hence so too is $MB \rightarrow MP$ (since $M$ was mono), therefore $MB \rightarrow MP \neq 0$.
\\ Therefore, $$0 \neq (EB \rightarrow EP)(y) = (EB \rightarrow EP)(EA' \rightarrow EB)(x) = (EA' \rightarrow EP)(x)= (EA \rightarrow EP)(EA' \rightarrow EA)(x)=0,$$
a contradiction. The first step is because $MB \rightarrow EB \rightarrow EP = MB \rightarrow MP \rightarrow EP$ (a naturality square), and $MP \rightarrow EP$ is monic.
\\
\end{proof}

\noindent Let $\mathcal{M}(\mathcal{A})$ be the full subcategory of $[\mathcal{A}, Ab]$ whose objects are the mono functors.
\\ $\mathcal{M}(\mathcal{A})$ is closed under taking subobjects, products, and essential extensions:
\begin{itemize}
    \item Let $E$ be a subfunctor of $F$ mono, so each component $EA \rightarrow FA$ is monic.
    \\ If $A \rightarrow A''$ is monic in $\mathcal{A}$, then $FA \rightarrow FA''$ is monic in $Ab$. Since $E \rightarrow F$ is a natural transformation, $$EA \rightarrow EA'' \rightarrow FA'' = EA \rightarrow FA \rightarrow FA'', \text{ which is monic.}$$
    \item Closure under products is easy: we prove a similar result in Theorem \ref{7.27}. (There we show the full subcategory of left-exact functors is closed under products.)
    \item We have just proven closure under essential extensions.
    \\
\end{itemize}

\noindent To generalise the situation, let $\mathcal{B}$ be a Grothendieck category with injective extensions, and let $\mathcal{M}$ be a full subcategory closed under taking subobjects, products, and essential extensions. Let us call the objects in $\mathcal{M}$ \textit{mono objects}.
\\ As an example, if $R$ is an integral domain, then $\mathcal{B} = R$-Mod is Grothendieck, and the subcategory $\mathcal{M}$ of torsion-free modules is closed under these three operations.

\begin{prop}
Every object $B \in \mathcal{B}$ has a maximal quotient object $B \rightarrow \mathcal{M} (B)$ in $\mathcal{M}$.
\end{prop}

\begin{proof}
Let $\mathcal{F}$ be the set of mono quotients of $\mathcal{B}$. Define $M(B)$ as the coimage of $B \xrightarrow{h} \Pi_{B' \in \mathcal{F}} B'$, where each component of $h$ is the obvious epic.
\\ Since $\mathcal{M}$ is closed under products and subobjects, and a coimage is a subobject of its
codomain, $M (B) \in \mathcal{M}$. By definition, the coimage is a quotient object of $B$, so it remains to see that it is maximal in $\mathcal{M}$.
\\ Given an epic $B \rightarrow B''$ with $B'' \in \mathcal{M}$, there is a map $M(B) \rightarrow B''$ such that

 \[
  \begin{tikzcd}[row sep=0.4em,column sep=2em]
    & M(B) \arrow{dd} \\
    B \arrow[ur] \arrow[dr] && \text{commutes.}\\
    & B'' \\
  \end{tikzcd}
 \]
Indeed, since $B''$ appears as a factor in the product $\Pi B'$, we have a projection map $\Pi B_i \xrightarrow{\pi} B''$. Let us take $$M(B) \rightarrow B'' = M(B) \rightarrow \Pi B' \xrightarrow{\pi} B''.$$
Then, as desired, $$B \rightarrow M(B) \rightarrow B'' = B \rightarrow M(B) \rightarrow \Pi B' \xrightarrow{\pi} B'' = B \rightarrow \Pi B' \xrightarrow{\pi} B'' = B \rightarrow B''.$$
\end{proof}

\begin{prop} \label{7.22}
Let $B \rightarrow M$ be any map, where $B \in \mathcal{B}, M \in \mathcal{M}$.
\\ There is a unique map $M(B) \rightarrow M$ such that
 \[
  \begin{tikzcd}[row sep=0.4em,column sep=2em]
    & M(B) \arrow{dd} \\
    B \arrow[ur] \arrow[dr] && \text{commutes.}\\
    & M \\
  \end{tikzcd}
 \]
 (This says that $B \rightarrow M(B)$ is a \textbf{reflection} of $B$ in $\mathcal{M}$.)
\end{prop}

\begin{proof}
Let $B \xrightarrow{c} B''$ be the coimage of $B \xrightarrow{b} M$, so $b$ factors as $B \xrightarrow{c} B'' \xrightarrow{d} M$.
\\ $B''$ is mono since $\mathcal{M}$ is closed under subobjects. By maximality of $M(B)$ among mono quotients, $B \xrightarrow{c} B''$ factors as $B \xrightarrow{q} M(B) \xrightarrow{u} B''$.
\\ Let us define $M(B) \xrightarrow{x} M$ as $du$.
\\ We clearly have $b = dc = duq = xq$.
\\ $x$ is the unique such map, because $q$ is epic: if also $b = x'q$, then $x'q = b = xq$ implies $x=x'$.
\\
\end{proof}

\noindent Given any $B' \rightarrow B$ in $\mathcal{B}$, we get a unique map $M(B') \rightarrow M(B)$ such that

\[ \begin{tikzcd}
B' \arrow{r} \arrow[swap]{d} & M(B') \arrow{d} \\
B \arrow{r} & M(B) & \text{commutes,}
\end{tikzcd} 
\]

\noindent by taking $B' \rightarrow M(B)$ to be the composite $B' \rightarrow B \rightarrow M(B)$ in 
 \[
  \begin{tikzcd}[row sep=0.3em,column sep=2em]
    & M(B') \arrow{dd}{\exists !} \\
    B' \arrow[ur] \arrow{dr}\\
    & M(B) \ . \\
  \end{tikzcd}
 \]
 
\noindent The uniqueness forces $M$ to be an additive functor $\mathcal{B} \rightarrow \mathcal{M}$.
\begin{itemize}
    \item $M(1_B) = 1_M(B)$, because $1_{M(B)}$ makes the diagram
\[ \begin{tikzcd}
B \arrow{r} \arrow[swap]{d}{1_B} & M(B) \arrow{d}{1_{M(B)}} \\
B \arrow{r} & M(B) & \text{commute.}
\end{tikzcd} 
\]
    \item $M(B' \xrightarrow{f} B \xrightarrow{g} B'') = M(g) \circ M(f)$, because $M(g) \circ M(f)$ makes the diagram
\[ \begin{tikzcd}
B' \arrow{r} \arrow[swap]{d}{g \circ f} & M(B') \arrow{d}{M(g) \circ M(f)} \\
B'' \arrow{r} & M(B'') & \text{commute. (The diagram expands as two commutative squares, one above the other.)}
\end{tikzcd} 
\]
    \item $M$ is additive -- each square in
\[ \begin{tikzcd}
B' \arrow{r}{q'} \arrow[swap]{d}{\Delta = \langle 1,1 \rangle } & M(B') \arrow{d}{\Delta = \langle 1,1 \rangle } \\
B' \oplus B' \quad  \arrow{r} \arrow[swap]{d}{[ f, g ]} &\quad M(B') \oplus M(B') \arrow{d}{[ Mf, Mg ] } \\
B \arrow{r}{q} & M(B') & \text{commutes, so the ``outer'' square commutes:}
\end{tikzcd} 
\]
$$[ Mf, Mg ] \circ [i_1 \circ q', i_2 \circ q'] = [[Mf, Mg] \circ i_1 \circ q', [ Mf, Mg ] \circ i_2 \circ q']  = [Mf \circ q', Mg \circ q'] = [q \circ f, q \circ g] = q \circ [f,g].$$
$$\langle q' \circ \pi_1, q' \circ \pi_2 \rangle \circ \langle 1,1 \rangle = \langle q' \circ \pi_1 \circ \langle 1,1 \rangle \, q' \circ \pi_2 \circ \langle 1,1 \rangle \rangle = \langle q' \circ 1,q' \circ 1 \rangle =  \langle 1 \circ q',1 \circ q' \rangle = \langle 1,1 \rangle q'.$$
\end{itemize}

\noindent (We have shown that $\mathcal{M}$ is a \textbf{reflective} subcategory of $\mathcal{B}$, and the functor $M: \mathcal{B} \rightarrow \mathcal{M}$ is a \textbf{reflector}.)
\\
\\

\noindent Let us call $T \in \mathcal{B}$ a \textbf{torsion} object if $Hom(T,N) = 0$ for each $N \in \mathcal{M}$.

\begin{prop}
$T$ is torsion iff $M(T) = 0$.
\end{prop}

\begin{proof}
Suppose $M(T) = 0$. Let $T \xrightarrow{k} N$ be any map, where $N \in \mathcal{M}$. By Proposition \ref{7.22} this map factors as $k = T \rightarrow 0 \rightarrow N = 0$.
\\

\noindent In the other direction, $Hom(T, M(T)) = 0$ means the obvious epic $T \rightarrow M(T)$ is zero, so $M(T) =  0$.
\\ (For instance, any map out of $M(T)$ must be zero, since $T \xrightarrow{0} M(T)$ is epic. This shows $M(T)$ is initial.)
\\
\end{proof}

\begin{prop} \label{7.24}
$Ker(B \rightarrow M(B))$ is the maximal torsion subobject of $B$.
\end{prop}

\begin{proof}
For any torsion object $T$ and map $T \rightarrow B$, the image of $T \rightarrow B$ is contained in $Ker(B \rightarrow M(B))$.
($T \rightarrow B \rightarrow M(B) = 0$, because Proposition \ref{7.22} says this map factors as $T \rightarrow M(T) \rightarrow M(B)$, and $M(T) = 0$.)
\\ Hence, as soon as we that $K=Ker(B \rightarrow M(B))$ is torsion, we are done: it is maximal as such.
\\

\noindent Let $B'' \in \mathcal{M}$. We show that any map $K \rightarrow B''$ is zero.
\\ We have an exact sequence $0 \rightarrow K \rightarrow B \rightarrow M(B) \rightarrow 0$.
\\ Form an injective envelope $B'' \rightarrow E$. $\mathcal{M}$ is closed under this operation, so $E \in \mathcal{M}$.
\\ Since $K \rightarrow B$ is monic and $E$ is injective, $K \rightarrow B'' \rightarrow E$ extends to a map $B \rightarrow E$.
\\ We obtain the following commutative diagram:

\[ \begin{tikzcd}
0 \arrow{r} &K \arrow{d} \arrow{r} &B \arrow{d} \arrow{r} &M(B) \arrow{dl} \arrow{r} &0 \\
& B'' \arrow{r} &E 
\end{tikzcd} 
\]
where $M(B) \rightarrow E$ is a map as in Proposition \ref{7.22}.
\\By commutativity of the diagram and exactness of the upper row, we have $$K \rightarrow B'' \rightarrow E = K \rightarrow B \rightarrow M(B) \rightarrow E = 0,$$
so $K \rightarrow B''$ is also zero since $ B'' \rightarrow E$ is monic.
\\
\end{proof}

\noindent In general, although $\mathcal{B}$ was abelian, $\mathcal{M}$ need not be: not every monic in $\mathbb{M}$ is realised as a kernel of a map in $\mathcal{M}$. For instance, in the situation where $R = \mathbb{Z}$, $\mathcal{B} = Ab$ and $\mathcal{M}$ is the subcategory of torsion-free abelian groups, $\mathcal{M}$ is not abelian because the monic $\mathbb{Z} \xrightarrow{2} \mathbb{Z}$ is not a kernel.
\\ (If it was a kernel of some $\mathbb{Z} \xrightarrow{f} B$, then $2f(1)=f(2)=0$, which forces $f(1)=0$ since $B$ is torsion free. So $f=0$. Then $\mathbb{Z} \xrightarrow{1} \mathbb{Z}$ would also factor through the kernel of $f$, which implies that $1$ is even: a contradiction.)
\\

\noindent What we \textit{can} do is go one level deeper to define a full subcategory $\mathcal{L}$ of $\mathcal{M}$ that will turn out to be abelian. In the case when our Grothendieck category is $[\mathcal{A}, Ab]$, $\mathcal{L}$ will be our key to proving the Mitchell embedding theorem.
\\

\noindent Let us call a subobject $M' \subset M \in \mathcal{M}$ to be \textbf{pure} if $M/M' \in \mathcal{M}$, where $M \rightarrow M/M'$ is the cokernel  of $M' \rightarrow M$. Let us call a mono object \textbf{absolutely pure} if whenever it appears as a subobject of a mono object, it is a pure subobject.
\\ Define $\mathcal{L}$ to be the full subcategory of absolutely pure objects.

\begin{lem} \label{inj abspure}
All injective mono objects are absolutely pure.
\end{lem}

\begin{proof}
Let $E \in \mathcal{M}$ be injective. Let $E \rightarrow F$ be monic, where $F \in \mathcal{M}$. We must show that $F/E \in \mathcal{M}$. Well, the extension $E \rightarrow F$ must be split, so $F$ is the direct sum of $E$ and $F/E$. In particular $F/E$ is a subobject of $F \in \mathcal{M}$, so $F/E \in \mathcal{M}$.
\\
\end{proof}

\begin{lem}
If $0 \rightarrow M_1 \rightarrow B \rightarrow M_2 \rightarrow 0$ is exact in $\mathcal{B}$ and $M_1, M_2 \in \mathcal{M}$, then $B \in \mathcal{M}$.
\end{lem}

\begin{proof}
Let $M_1 \rightarrow E$ be an injective envelope.
\\ We have $E \in \mathcal{M}$, hence $E \oplus M_2 \in \mathcal{M}$.
\\ Since $E$ is injective, $B \rightarrow E$ extends to a map $M_1 \rightarrow E$. Once we see that the obvious map $B \xrightarrow{m} E \oplus M_2$ is monic, we will be done.
\\ Suppose $f,g$ are maps $A \rightarrow B$ with $mf=mg$. It is enough to see that $d:=f-g = 0$.
\\ Since $md=0$, $A \xrightarrow{d} B \rightarrow M_2 =A \xrightarrow{d} B \rightarrow E = 0$. In particular, $d$ factors through the kernel of $B \rightarrow M_2$ as $d = A \rightarrow M_1 \rightarrow B$. Now
$$A \rightarrow M_1 \rightarrow E = A \rightarrow M_1 \rightarrow B \rightarrow E = A \xrightarrow{d} B \rightarrow E = 0,$$
but $M_1 \rightarrow E$ is monic, so $A \rightarrow M_1 = 0$, and finally $d = A \rightarrow M_1 \rightarrow B= 0$.
\\
\end{proof}

\begin{lem} \label{7.26}
A pure subobject of an absolutely pure object is absolutely pure.
\end{lem}

\begin{proof}
Let $A$ be absolutely pure, $P \rightarrow A$ a pure subobject, and $P \rightarrow M$ a monic, where $M \in \mathcal{M}$. We must show that $M/P \in \mathcal{M}$.
\\ Make a pushout diagram
\[ \begin{tikzcd}[row sep=0.8em,column sep=1em]
P \arrow{r} \arrow[swap]{d} & A \arrow{d} \\
M \arrow{r} & R
\end{tikzcd}
\]
and extend it to an exact commutative diagram
\[ \begin{tikzcd}[row sep=0.8em,column sep=1em]
 &0 \arrow{d} &0 \arrow{d} &0 \arrow{d}\\
0 \arrow{r} & P \arrow{r} \arrow[swap]{d} & A \arrow{r} \arrow{d} &A/P \arrow{d} \arrow{r} & 0 \\
0 \arrow{r} & M \arrow{r} \arrow[swap]{d} & R \arrow{r} \arrow{d} &R/M \arrow{d} \arrow{r} & 0 \\
0 \arrow{r} & M/P \arrow{d} \arrow{r} & R/A \arrow{d} \arrow{r} & 0 \\
&0 &0 \\
\end{tikzcd}
\]
by noting that $M/P \rightarrow R/A$ and $A/P \rightarrow R/M$ are both isomorphisms.
\\
%%%%%%%%%%%%%%%%%%%%%%%%%%%%%
\noindent Since $M$ and $R/M \cong A/P$ are mono, $R$ is mono. Hence $R/A$ is mono, hence $M/P$ is mono, as required.
\\
\end{proof}

\begin{thm} \label{7.27}
A mono functor $M \in [\mathcal{A}, Ab]$ is absolutely pure iff it is left-exact.
\end{thm}

\begin{proof}
First, we prove a claim.
\begin{itemize}
    \item[\textbf{Claim:}] A subfunctor of a left-exact functor is pure iff it is left-exact.
    \item[\textbf{Proof:}]
    Let $0 \rightarrow M \rightarrow E \rightarrow F \rightarrow 0$ be exact in $[\mathcal{A}, Ab]$, where $E$ is left-exact. We must show $M$ is left-exact.
\\ Let $0 \rightarrow A' \rightarrow A \rightarrow A''$ be exact in $\mathcal{A}$.
We have a commutative diagram

\[ \begin{tikzcd}[row sep=0.8em,column sep=1em]
 &0 \arrow{d} &0 \arrow{d} &0 \arrow{d}\\
0 \arrow{r} & MA' \arrow{r} \arrow[swap]{d} & MA \arrow{r} \arrow{d} &MA'' \arrow{d}  \\
0 \arrow{r} & EA' \arrow{r} \arrow[swap]{d} & EA \arrow{r} \arrow{d} &EA'' \\
0 \arrow{r} & FA' \arrow{d} \arrow{r} & FA \arrow{d} \\
&0 &0 \\
\end{tikzcd}
\]
whose columns are exact since the evaluation functor $[\mathcal{A}, Ab] \rightarrow Ab$ for each of $A',A, \text{ and } A''$ is exact.
\\ The middle row is exact since $E$ is left-exact, so the hypothesis of Lemma \ref{2.64} is satisfied.
\\ Hence, $F$ is mono iff $M$ is left-exact. The claim is proven.
\end{itemize}

\noindent Now, suppose we have a mono functor $M$. Take an injective envelope $M \rightarrow E$. We know $E$ is left-exact (it is injective and mono, hence exact) and absolutely pure (by Lemma \ref{inj abspure}).
\\ If $M$ is absolutely pure, then $M \rightarrow E$ is pure, so the claim implies $M$ is left-exact.
\\ Conversely, if $M$ is left-exact, the claim implies $M \rightarrow E$ is pure, so we finish by Lemma \ref{7.26}.
\end{proof}

\noindent Recall that in the general setting we have a Grothendieck category $\mathcal{B}$, a full subcategory $\mathcal{M}$ of $\mathcal{B}$ closed under taking subobjects, products, and essential extensions, and a full subcategory $\mathcal{L}$ of $\mathcal{M}$ consisting of the absolutely pure objects.
\\

\noindent Given $M \in \mathcal{M}$ and $R \in \mathcal{L}$, a map $M \rightarrow R$ is a  \textbf{reflection} of $M$ in $\mathcal{L}$ if for every map $M \rightarrow L$ where $L \in \mathcal{L}$, there is a unique map $R \rightarrow L$ such that
 \[
  \begin{tikzcd}[row sep=0.4em,column sep=2em]
    & R \arrow{dd} \\
    M \arrow[ur] \arrow[dr] && \text{commutes.}\\
    & L \\
  \end{tikzcd}
 \]

\begin{thm}[Recognition Theorem] \label{7.28}
If the sequence $0 \rightarrow M \rightarrow R \rightarrow T \rightarrow 0$ is exact in $\mathcal{B}$ for $M$ mono, $R$ absolutely pure, and $T$ torsion, then $M \rightarrow R$ is a reflection of $M$ in $\mathcal{L}$.
\end{thm}

\begin{proof}
Given $L \in \mathcal{L}$ and $M \xrightarrow{m} L$, let $L\xrightarrow{l} E$ be an injective envelope and $E \xrightarrow{c} F = Cok(L\xrightarrow{l} E)$.
\\ We obtain a commutative diagram with exact rows

\[ \begin{tikzcd}[row sep=1.1em,column sep=1.5em]
0 \arrow{r} &M \arrow{d}{m} \arrow{r}{i} &R \arrow{d}{r} \arrow{r} &T \arrow{d} \arrow{r} &0 \\
0 \arrow{r} &L  \arrow{r}{l} &E \arrow{r}{c} &F \arrow{r} &0 \\
\end{tikzcd} 
\]
as follows:
\\ We already have exact rows, and the vertical map $M \xrightarrow{m} L$. Since $M \xrightarrow{i} R$ is monic and $E$ is injective, $M \rightarrow L \rightarrow E$ extends to a map $R \xrightarrow{r} E$. Finally, we get a map $T \rightarrow F$ because the map $R \xrightarrow{r} E \xrightarrow{c} F$ factors through the cokernel $R \rightarrow T$ of $M \xrightarrow{i} R$. ($cri = clm = 0m = 0$.)
\\

\noindent $E$ is mono by Lemma \ref{7.12}, so $F$ is mono by absolute purity of $L$. Hence $T \rightarrow F = 0$, by definition of $T$ torsion. Then $R \xrightarrow{r} E \xrightarrow{c} F = 0$, so $r$ factors through the kernel $L$ of $c$: there is some $R \xrightarrow{u} L$ with $lu=r$.
\\ Then $lui=ri=lm$, but $l$ is monic so we have found that
 \[
  \begin{tikzcd}[row sep=0.4em,column sep=2em]
    & R \arrow{dd}{u} \\
    M \arrow{ur}{i} \arrow{dr}{m} && \text{commutes.}\\
    & L \\
  \end{tikzcd}
 \]
 It remains to see that $u$ is the unique such map.
 \\ If we have $u, u' : R \rightarrow L$ with $ui=m=u'i$, then $d:=u-u'$ factors through the cokernel $R \rightarrow T$ of $i$,
 $$R \xrightarrow{d} L = R \rightarrow T \rightarrow L = 0,$$
where the last equality follows since $T$ is torsion.
\\
\end{proof}

\begin{thm}[Construction Theorem] \label{7.29}
For every $M \in \mathcal{M}$ there is a monic $M \rightarrow R$ which is a reflection of $M$ in $\mathcal{L}$.
\end{thm}

\begin{proof}
Let $M \rightarrow E$ be an injective envelope. In particular $E$ is absolutely pure.
\\ Construct an exact commutative diagram
\[ \begin{tikzcd}[row sep=0.8em,column sep=1em]
 &0 \arrow{d} &0 \arrow{d} &0 \arrow{d}\\
0 \arrow{r} & M \arrow{r} \arrow[swap]{d} & R \arrow{r} \arrow{d} &T \arrow{d} \arrow{r} &0 \\
0 \arrow{r} & M \arrow{r} \arrow[swap]{d} & E \arrow{r} \arrow{d} &F \arrow{r} \arrow{d} &0 \\
 & 0 \arrow{r} & M(F) \arrow{d} \arrow{r} &M(F) \arrow{d} \arrow{r} &0  \\
&&0 &0 \\
\end{tikzcd}
\]
by starting with the middle row (constructed from the monic $M \rightarrow E$), then the right column (constructed from the epic $F \rightarrow M(F)$), the the bottom row, then the middle column (constructed from $E \rightarrow M(F)$, epic as the composition of two epics), then the top row. The top row is the only part that is not exact by construction; the map $M \rightarrow R$ exists because $M \rightarrow E$ factors through the kernel $R \rightarrow E$ of $E \rightarrow M(F)$.
\\ The top row is exact by Lemma \ref{2.65}.
\\

\noindent We finish simply by applying Theorem \ref{7.28} to the top row, since $M$ is mono, $T$ is torsion by Proposition \ref{7.24}, and $R$ is absolutely pure ($R \rightarrow E$ is pure since $M(F) \in \mathcal{M}$, and $E$ is absolutely pure).
\\
\end{proof}

\noindent $\mathcal{L}$ is seen to be a reflective subcategory of $\mathcal{M}$, in exactly the same way we showed that $\mathcal{M}$ was a reflective subcategory of $\mathcal{B}$: choosing a reflection $M \rightarrow R(M)$ in $\mathcal{L}$ for each $M \in \mathcal{M}$ yields an additive functor $R: \mathcal{M} \rightarrow \mathcal{L}$.
\\ (Reflections are unique up to isomorphism.)
\\

\begin{thm} \label{7.31}
$\mathcal{L}$ is abelian, and every object has an injective envelope.
\end{thm}

\begin{proof}
We check the axioms.
\begin{itemize}
    \item[A0.] The constantly zero functor is a zero object.
    \item[A1, A1*.] For $M \in \mathcal{M}$, we have $M \in \mathcal{L}$ iff $M \rightarrow R(M)$ is an isomorphism.
    \\ (If $M \in \mathcal{L}$ then $M \xrightarrow{1} M$ is a reflector; conversely if $M \cong R(M) \in \mathcal{L}$ then $M \in \mathcal{L}$.)
    \\ Since $R$ is an additive functor, it preserves direct sums: given $N,N' \in \mathcal{L}$, we have
    $$R(N \oplus N') \cong R(N) \oplus R(N') \cong N \oplus N',$$
    so $N \oplus N' \in \mathcal{L}$.
    
    \item[A2.] By Lemma \ref{7.26}, the $\mathcal{B}$-kernel of an $\mathcal{L}$-map $L \rightarrow L'$ is in $\mathcal{L}$, so $\mathcal{L}$ has kernels. Indeed, write $K \rightarrow L$ for the $\mathcal{B}$-kernel. Then $K \in \mathcal{L}$, since $L \in \mathcal{L}$ and $K \rightarrow L$ is a pure subobject (as K/L is just $L' \in \mathcal{L}$).
    \\ In fact, an $\mathcal{L}$-map s an $\mathcal{L}$-monic iff it is a $\mathcal{B}$-monic. (Both conditions are equivalent to the kernel being zero.)
    
    \item[A3.] Let $L \rightarrow L'$ be an $\mathcal{L}$-monic, and let $L' \rightarrow L'/L$ be its $\mathcal{B}$-cokernel. Since $L$ is absolutely pure and $L' \in \mathcal{M}$, $M:=L'/L \in \mathcal{M}$. Therefore it has a reflection $M \rightarrow R(M)$ in $\mathcal{L}$.
    \\ Now $L \rightarrow L'$ is the $\mathcal{B}$-kernel of $L' \rightarrow M$, and hence the $\mathcal{B}$-kernel of $L' \rightarrow R(M) = L' \rightarrow M \rightarrow R(M)$, since $M \rightarrow R(M)$ is monic.
    \\ Of course, the $\mathcal{B}$-kernel of an $\mathcal{L}$-map \textit{is} its $\mathcal{L}$-kernel, so $L \rightarrow L'$ the $\mathcal{L}$-kernel of $L' \rightarrow R(M))$.
    
    \item[A2*.] Let $L \rightarrow L'$ be an $\mathcal{L}$-map. Take the $\mathcal{B}$-cokernel $L' \rightarrow F$.
    \\ Then $L' \rightarrow F \rightarrow M(F) \rightarrow R(M(F))$ is an $\mathcal{L}$-cokernel:
    \\ Certainly $L \rightarrow L' \rightarrow F \rightarrow M(F) \rightarrow R(M(F)) =0$, since $L \rightarrow L' \rightarrow F =0$. Now suppose $L \rightarrow L' \rightarrow N =0$. Then $L' \rightarrow N$ factors uniquely through the cokernel as $L' \rightarrow F \xrightarrow{\exists !} N$. In turn, $F \rightarrow N$ factors uniquely through $M(F)$ as $F \rightarrow M(F) \xrightarrow{\exists !} N$, and in turn still, $M(F) \rightarrow N$ factors uniquely through $R(M(F))$ as $M(F) \rightarrow R(M(F)) \xrightarrow{\exists !} N$.
    \\ Taken altogether, $L' \rightarrow N$ factors uniquely through $L' \rightarrow F \rightarrow M(F) \rightarrow R(M(F))$.
    %%%%%%%%%%%%%%%%%%%%%%%%%%%
    \item[A3*.] The above shows that an $\mathcal{L}$-map $L \rightarrow L'$ is an $\mathcal{L}$-epic iff the $\mathcal{B}$-cokernel of $L \rightarrow L'$ is torsion.
    \\ Let $L \rightarrow L'$ be an $\mathcal{L}$-epic. Take its $\mathcal{B}$-image $M \rightarrow L'$ to get an exact sequence $ 0 \rightarrow M \rightarrow L' \rightarrow T \rightarrow 0$ in $\mathcal{B}$. We have just remarked that $T$ must be torsion; furthermore $M$ is mono as a subobject of $L' \in \mathcal{M}$, so we may apply Theorem \ref{7.28} to this sequence to deduce that $L' \cong R(M)$.
    \\ Therefore, write $K \rightarrow L = Ker(L \rightarrow M)$; then the $\mathcal{B}$-cokernel of $K \rightarrow L$ is $L \rightarrow M$.
    \\  We know that the $\mathcal{L}$-cokernel of $K \rightarrow L$ must be the $\mathcal{B}$-cokernel postcomposed with a reflection down to $\mathcal{M}$ and then another reflection down to $\mathcal{L}$, but this is just $L \rightarrow R(M) = L \rightarrow L'$.
    \\ We have exhibited $L \rightarrow L'$ as an $\mathcal{L}$-cokernel!
\end{itemize}
\noindent Therefore $\mathcal{L}$ is abelian.
\\Let us see that every object in $\mathcal{L}$ has an injective envelope. Since monics are the same in $\mathcal{B}$ and in $\mathcal{L}$, if $E$ is a $\mathcal{B}$-injective envelope of an $\mathcal{L}$-object, then it is injective in $\mathcal{L}$.
\\ (To spell this out: take an injective envelope $L \rightarrow E$ in $\mathcal{B}$. $E \in \mathcal{B}$ is injective and mono, hence absolutely pure. $L \rightarrow E$ is still an injective essential extension in $\mathcal{L}$.)
\\
\end{proof}

\noindent Finally, let us return to the case of the Grothendieck category $[\mathcal{A}, Ab]$. Just as $\mathcal{M}(\mathcal{A})$ was the full subcategory of mono functors, let us define $\mathcal{L}(\mathcal{A})$ to be the full subcategory of left-exact functors.
\\ Theorems \ref{7.27} and \ref{7.31} say that $\mathcal{L}(\mathcal{A})$ is an abelian category wtih injective envelopes. The Yoneda embedding $H: \mathcal{A} \rightarrow [\mathcal{A}, Ab]$ factors through $\mathcal{L}(\mathcal{A})$, precisely because each $H^A=Hom(A,-)$ is left-exact.
\\

\begin{thm} \label{7.32}
$\mathcal{L}(\mathcal{A})$ is complete and has an injective cogenerator.
\end{thm}

\begin{proof}
Products in $\mathcal{L}(\mathcal{A})$ are just products in $[\mathcal{A},Ab]$, because the product of left-exact functors of $[\mathcal{A},Ab]$ is left-exact: Suppose we have a family $\{F_i\}_I$ in $\mathcal{L}(\mathcal{A})$. Let $0 \rightarrow A' \rightarrow A \rightarrow A$ be exact in $\mathcal{A}$. Then $0 \rightarrow F_iA' \rightarrow F_iA \rightarrow F_A''$ is exact for each $i$. \\Taking the product of these sequences in $Ab$ yields $0 \rightarrow \Pi (F_iA') \rightarrow \Pi(F_iA) \rightarrow \Pi(F_i A'')$, exact in $Ab$, but of course this last sequence is just $0 \rightarrow (\Pi F_i)A' \rightarrow (\Pi F_i)A \rightarrow (\Pi F_i)A''$.
\\
%%%%%%%%%%%%%%%%%%%
\noindent In particular the product of all the representables $\{H^A\}_{A \in \mathcal{A}}$ is also left-exact, and since this was a generator for $[\mathcal{A}, Ab]$ (Theorem \ref{5.35}), it is a generator for $\mathcal{L}(\mathcal{A})$.
\\ By Proposition \ref{3.37}, $\mathcal{L}(\mathcal{A})$ has an injective cogenerator.
%%%%%%%%%%%%%%%%%%%%%%
\\
\end{proof}

\begin{thm}
$H: \mathcal{A}^{op} \rightarrow \mathcal{L}(\mathcal{A})$ is an exact full embedding.
\end{thm}

\begin{proof}
We know $H$ is a full embedding (Theorem \ref{5.36}); it remains to show $H$ is exact.
\\ Let $0 \rightarrow A' \rightarrow A \rightarrow A'' \rightarrow 0$ be exact in $\mathcal{A}$. We must show $0 \rightarrow H^{A''} \rightarrow H^{A} \rightarrow H^{A'} \rightarrow 0$ is exact in $\mathcal{L}(\mathcal{A})$.
\\ This is the case iff the sequence
$0 \rightarrow Hom(H^{A'},E) \rightarrow Hom(H^{A},E) \rightarrow Hom(H^{A''},E) \rightarrow 0$ is exact in $Ab$ for an injective cogenerator $E$ in $\mathcal{L}(\mathcal{A})$.
\\ ($E$ is injective, so $Hom(-,E)$ is exact; $E$ is a cogenerator, so $Hom(-,E)$ is an embedding; apply Corollary \ref{exemb}.)\\
\noindent That last sequence is isomorphic by the Yoneda Lemma to $0 \rightarrow EA' \rightarrow EA \rightarrow EA'' \rightarrow 0$, and this sequence is always exact iff $E$ is an exact functor.
\\ This is indeed the case: $E$ is right-exact by Lemma \ref{7.11}, and left-exact since it lies in $\mathcal{L}(\mathcal{A})$.
\\
\end{proof}

\begin{thm}[Freyd-Mitchell]
Every abelian category is fully abelian.
\end{thm}

\begin{proof}
The Yoneda embedding $H: \mathcal{A}^{op} \rightarrow \mathcal{L}(\mathcal{A})$ provides an exact full embedding into a complete abelian category with an injective cogenerator.
\\ We may of course view this as a functor $H: \mathcal{A} \rightarrow \mathcal{L}(\mathcal{A})^{op}$. This is an exact full embedding into a cocomplete abelian category with a projective generator. Now apply Theorem \ref{Mitch}.
\end{proof}

\begin{cor}
For every small abelian category $\mathcal{A}$ there is a ring $R$ and an exact full embedding
\\ $\mathcal{A} \rightarrow R$-Mod.
\end{cor}


\newpage
\begin{thebibliography}{100}
%USE%HARVARD%STYLE
\raggedright

\bibitem{Freyd64}
Freyd, P. (1964) \textit{Abelian Categories – An Introduction to the theory of functors}. New York, Harper and Row.

\bibitem{Leinster14}
Leinster, T. (2014) \textit{Basic Category Theory}, Cambridge Studies in Advanced
Mathematics, Vol. 143. Cambridge, Cambridge University Press.

\end{thebibliography}
\end{document}